\documentclass[11pt]{nyjm}
\input{xy}
\xyoption{all}

\title[Residues on and dualizing sheaves of arithmetic surfaces]{An explicit approach to residues on and dualizing sheaves of arithmetic surfaces} 

\author{Matthew Morrow}
\address{University of Chicago,\newline Dept. of Mathematics,\newline 5734 S. University Avenue,\newline Chicago, IL, 60637,\newline USA} 
\email{mmorrow@math.uchicago.edu\newline \url{http://math.uchicago.edu/~mmorrow}}   
\thanks{Much of the work in this paper was done while I was the recipient of the Cecil King Travel prize, through the London Mathematical Society; I would like to thank the Cecil King Foundation for their generosity.\newline\indent This paper has been published by the New York Journal of Mathematics, vol. 16, 2010, pp. 575 -- 627, and is available at \url{http://nyjm.albany.edu/j/2010/16-25.html}.\newline\indent The sequel to this paper is \cite{Morrow2010}. \newline\indent The earlier version of the paper, entitled ``An explicit approach to residues on and {\em canonical} sheaves of arithmetic surfaces'', is now redundant.} 



\usepackage{amsmath, amscd, amsfonts, amssymb, latexsym, url}
\usepackage{graphicx}

\newtheorem{theorem}{Theorem}[section]
\newtheorem{lemma}[theorem]{Lemma}
\newtheorem{corollary}[theorem]{Corollary}
\newtheorem{proposition}[theorem]{Proposition}

\theoremstyle{definition}
\newtheorem{definition}[theorem]{Definition}
\newtheorem{remark}[theorem]{Remark}
\newtheorem{example}[theorem]{Example}

\renewcommand{\labelenumi}{(\roman{enumi})}

\renewcommand{\hat}{\widehat}
\renewcommand{\tilde}{\widetilde}
\newcommand{\bb}{\mathbb}
\newcommand{\into}{\hookrightarrow}
\newcommand{\onto}{\to\!\!\!\!\!\to}
\newcommand{\To}{\longrightarrow}
\newcommand{\sub}[1]{{\mbox{\scriptsize #1}}}
\newcommand{\res}{\overline}
\newcommand{\roi}{\mathcal{O}}
\newcommand{\al}{\alpha}
\newcommand{\ep}{\varepsilon}
\newcommand{\mult}[1]{#1^{\!\times}}
\renewcommand{\frak}{\mathfrak}
\newcommand{\cal}{\mathcal}
\newcommand{\dblecurly}[1]{\{\!\{ #1 \}\!\}}
\newcommand{\dd}[2]{\frac{\partial #1}{\partial #2}}
\newcommand{\omegaa}{\mbox{\begin{large}$\omega$\end{large}}}
\newcommand{\isoto}{\stackrel{\simeq}{\to}}
\DeclareMathOperator{\Frac}{Frac}
\DeclareMathOperator{\Gal}{Gal}
\DeclareMathOperator{\Spec}{Spec}
\DeclareMathOperator{\RES}{res}
\DeclareMathOperator{\RRES}{Res}
\DeclareMathOperator{\Hom}{Hom}
\DeclareMathOperator{\Tr}{Tr}
\DeclareMathOperator{\Tors}{Tors}

\begin{document} 

\begin{abstract}  
We develop a theory of residues for arithmetic surfaces, establish the reciprocity law around a point, and use the residue maps to explicitly construct the dualizing sheaf of the surface. These are generalisations of known results for surfaces over a perfect field. In an appendix, explicit local ramification theory is used to recover the fact that in the case of a local complete intersection the dualizing and canonical sheaves coincide.
\end{abstract} 
\maketitle
\tableofcontents


\section{Introduction}
This paper studies arithmetic surfaces using two-dimensional local fields associated to the scheme, and thus further develops the ad\`elic approach to higher dimensional algebraic and arithmetic geometry. We study residues of differential forms and give an explicit construction of the relative dualizing sheaf. While considerable work on these topics has been done for varieties over perfect fields by Lipman, Lomadze, Parshin, Osipov, Yekutieli, et al., the arithmetic case has been largely ignored. After summarising the contents of the paper, we discuss its relation to this earlier work and provide references.

In section \ref{section_local_relative_residues} we consider a two-dimensional local field $F$ of characteristic zero and a fixed local field $K\le F$. We introduce a relative residue map \[\RRES_F:\Omega_{F/K}^\sub{cts}\to K,\] where $\Omega_{F/K}^\sub{cts}$ is a suitable space of `continuous' relative differential forms. In the case $F\cong K((t))$, this is the usual residue map; but if $F$ is of mixed characteristic, then our residue map is new (though essentially contained in Fesenko's ad\`elic analysis and Osipvov's study of surfaces: see the discussion below). Functoriality of the residue map is established with respect to a finite extension $F'/F$, i.e. \[\RRES_F\Tr_{F'/F}=\RRES_{F'}.\]

In section \ref{section_reciprocity_for_rings} we prove the reciprocity law for two-dimensional local rings, justifying our definition of the relative residue map for mixed characteristic fields. For example, suppose $A$ is a characteristic zero, two-dimensional, normal, complete local ring with finite residue field, and fix the ring of integers of a local field $\roi_K\le A$. To each height one prime $y\lhd A$ one associates the two-dimensional local field $\Frac\hat{A_y}$ and thus obtains a residue map $\RRES_y:\Omega_{\Frac A/K}^1\to K$. We prove \[\sum_y\RRES_y\omega=0\] for all $\omega\in\Omega_{\Frac A/K}^1$. The next section restates these results in the geometric language.

Section \ref{section_explicit_dualizing_sheaf_of_arith_surface} is independent of the main results of sections \ref{section_reciprocity_for_rings} and \ref{section_reciprocity_for_surfaces}, using the local residue maps for a different purpose, namely to explicitly construct the relative dualizing sheaf of an arithmetic surface. See subsection \ref{subsection_explicit_grothendieck_duality} below for a reminder on dualizing sheaves. Let $\roi_K$ be a Dedekind domain of characteristic zero with finite residue fields; its field of fractions is $K$. Let $\pi:X\to S=\Spec\roi_K$ be an arithmetic surface (more precisely, $X$ is normal and $\pi$ is flat and projective, with the generic fibre being a smooth curve). To each closed point $x\in X$ and integral curve $y\subset X$ containing $x$, our local residue maps define $\RRES_{x,y}:\Omega_{K(X)/K}^1\to K_{\pi(x)}\;(=\pi(x)\mbox{-adic completion of }K)$, and we prove the following:

\begin{theorem}
The dualizing sheaf $\omegaa_\pi$ of $\pi:X\to S$ is explicitly given by, for open $U\subseteq X$, \begin{align*}\omegaa_{\pi}(U)=\{\omega\in\Omega_{K(X)/K}^1:&\RRES_{x,y}(f\omega)\in\hat{\roi}_{K,\pi(x)}\mbox{ for}\\ &\mbox{all } x\in y\subset U\mbox{ and }f\in\roi_{X,y}\}\end{align*} where $x$ runs over all closed points of $X$ inside $U$ and $y$ runs over all curves containing $x$.
\end{theorem}

Appendix \ref{section_finite_morphisms_etc} establishes a local ramification result, generalising a classical formula for the different of an extension of local fields. Let $B$ be a Noetherian, normal ring of characteristic zero, and \[A=B[T_1,\dots,T_m]/\langle f_1,\dots,f_m\rangle\] a normal, complete intersection over $B$ which is a finitely generated $B$-module. Letting $J\in A$ be the determinant of the Jacobian matrix, we prove that \[\{x\in F:\Tr_{F/M}(xA)\subseteq B\}=J^{-1}A.\] (see theorem \ref{theorem_different_and_jacobian_in_general} for the more precise statement). In other words, the canonical and dualizing sheaves of $A/B$ are the same. The proof reduces to the case when $A$, $B$ are complete discrete valuation rings with an inseparable residue field extension; for more on the ramification theory of complete discrete valuation fields with imperfect residue field, see \cite{Abbes2002} \cite{Abbes2003} \cite{Xiao2007} \cite{Xiao2008a} \cite{Xiao2008b} \cite{Zhukov2000} \cite{Zhukov2003}. From this result one can easily recover the fact that if the above arithmetic surface $X\to S$ is a local complete intersection, then its canonical and dualizing sheaves coincide.

As much for author's benefit as that of the reader, let us now say a few words about the relation of this work to previous results of others:

\subsection{An introduction to the higher ad\`elic method}
A two-dimensional local field is a compete discrete valuation field whose residue field is a local field (e.g. $\bb{Q}_p((t))$); for an introduction to such fields, see \cite{IHLF}. If $A$ is a two-dimensional domain, finitely generated over $\bb{Z}$, with field of fractions $F$ and $0\lhd\frak{p}\lhd\frak{m}\lhd A$ is a chain of primes in $A$, then consider the following sequence of localisations and completions:
\[\begin{array}{ccccccccccc}
A&\!\leadsto\!&A_{\frak{m}}&\!\leadsto\!&\hat{A_{\frak{m}}}&\!\leadsto\!&\left(\hat{A_{\frak{m}}}\right)_{\frak{p}'}&\!\leadsto\!&\hat{\left(\hat{A_{\frak{m}}}\right)_{\frak{p}'}}&\!\leadsto\!&\mbox{Frac}\left(\hat{\left(\hat{A_{\frak{m}}}\right)_{\frak{p}'}}\right)\\
&&&&&&&&\parallel&&\parallel\\
&&&&&&&&A_{\frak{m},\frak{p}}&&F_{\frak{m},\frak{p}}
\end{array}\]
which we now explain in greater detail. It follows from the excellence of $A$ that $\frak{p}':=\frak{p}\hat{A_{\frak{m}}}$ is a radical ideal of $\hat{A_{\frak{m}}}$; we may localise and complete at $\frak{p}'$ and again use excellence to deduce that $0$ is a radical ideal in the resulting ring i.e. $A_{\frak{m},\frak{p}}$ is reduced. The total field of fractions $F_{\frak{m},\frak{p}}$ is therefore isomorphic to a finite direct sum of fields, and each is a two-dimensional local field.

Geometrically, if $X$ is a two-dimensional, integral scheme of finite type over $\Spec\bb{Z}$ with function field $F$, then to each closed point $x\in X$ and integral curve $y\subset X$ which contains $x$, one obtains a finite direct sum of two-dimensional local fields $F_{x,y}$. Two-dimensional ad\`elic theory aims to study $X$ via the family $(F_{x,y})_{x,y}$, in the same way as one studies a curve or number field via its completions. Analogous constructions exist in higher dimensions. Useful references are \cite{Hubl1996} \cite[\S1]{Parshin1983}.

\subsection{The classical case of a curve over a perfect field}
This paper is based closely on similar classical results for curves and it will be useful to give a detailed account of that theory.

\subsubsection*{Smooth curves}
Firstly, let $C$ be a smooth, connected, projective curve over a perfect field $k$ (of finite characteristic, to avoid complications with differential forms). We follow the discussion in \cite[III.7.14]{Hartshorne1977}. For each closed point $x\in C$ one defines the residue map $\RRES_x:\Omega_{K(C)/k}^1\to k$, and one then proves the reciprocity law \[\sum_{x\in C_0}\RRES_x(\omega)=0\] for all $\omega\in\Omega_{K(C)/k}^1$. Consider $\Omega_{K(C)/k}^1$ as a constant sheaf on $C$; then \[0\to\Omega_{C/k}^1\to\Omega_{K(C)/k}^1\to\Omega_{K(C)/k}^1/\Omega_{C/k}^1\to 0\] is a flasque resolution of $\Omega_{C/k}^1$, and the corresponding long exact sequence of \v{C}ech cohomology is \begin{equation}0\to\Omega_{C/k}^1(C)\to\Omega_{K(C)/k}^1\to\bigoplus_{x\in C_0}\frac{\Omega_{K(C)/k}^1}{\Omega_{\roi_{C,x}/k}^1}\to H^1(C,\Omega_{C/k}^1)\to 0.\label{1}\end{equation} Now, the map $\sum_x\RRES_x:\bigoplus_{x\in C_0}\Omega_{K(C)/k}^1/\Omega_{\roi_{C,x}/k}^1\to k$ vanishes on the image of $\Omega_{K(C)/k}^1$ (by the reciprocity law), and so induces \[\mbox{tr}_{C/k}:H^1(C,\Omega_{C/k}^1)\to k,\] which is the trace map of $C/k$ with respect to the dualizing sheaf $\Omega_{C/k}^1$.

Moreover, duality of $C$ may be interpreted (and proved) ad\`elically as follows; see \cite[II.\S8]{Serre1988}. For each $x\in C_0$, let $K(C)_x$ be the completion of $K(C)$ at the discrete valuation $\nu_x$ associated to $x$, and let \[\bb{A}_C=\{(f_x)\in\prod_{x\in C_0} K(C)_x:\,\nu_x(f_x)\ge0\mbox{ for all but finitely many }x\}\] be the ad\`elic space of $C$. Also, let \[\bb{A}(\Omega_{C/k}^1)=\{(\omega_x)\in\prod_{x\in C_0}\Omega_{K(C)_x/k}^1:\,\nu_x(\omega_x)\ge0\mbox{ for all but finitely many }x\}\] be the differential ad\`elic space of $C$. Then, under the pairing \[\bb{A}_C\times\bb{A}(\Omega_{C/k}^1)\to k,\quad((f_x),(\omega_x))\mapsto\sum_{x\in C_0}\RRES_x(f_x\omega_x),\] the orthogonal complement of $\bb{A}(\Omega_{C/k}^1(D))$ is \[\bb{A}(\Omega_{C/k}^1(D))^{\perp}=\bb{A}_C(D).\] Here $D$ is a divisor on $C$, and $\bb{A}_C(D)$ (resp. $\bb{A}(\Omega_{C/k}^1(D))$) is the subgroup of $\bb{A}_C$ (resp. $\bb{A}(\Omega_{C/k}^1$) for which $\nu_x(f_x)\ge -\nu_x(D)$ (resp. $\nu_x(\omega_x)\ge\nu_x(D)$) for all $x$. Moreover, the global elements, embedded diagonally, are self-dual: \[K(C)^{\perp}=\Omega_{K(C)/k}^1.\] The exact sequence (\ref{1}) generalises to the twisted sheaf $\Omega_{C/k}^1(D)$, and thereby provides an isomorphism \[\bb{A}(\Omega_{C/k}^1)/(\Omega_{K(C)/k}^1+\bb{A}(\Omega_{C/k}^1(D)))\cong H^1(C,\Omega_{C/k}^1(D));\] combining this with the aforementioned ad\`elic dualities yields the non-degenerate pairing \[\cal{L}(D)\times H^1(C,\Omega_{C/k}^1(D))\to k,\] where \[\cal{L}(D):=K(C)\cap\bb{A}_C(D)=\{f\in K(C):\nu_x(f)\ge -\nu_x(D)\mbox{ for all }x\in C_0\}.\] This is exactly duality of $C/k$.

\subsubsection*{Singular curves}
Secondly, suppose that $C$ is allowed to have singularities; we now follow \cite[IV.\S3]{Serre1988}. One may still define a residue map at each closed point $x$; in fact, if $\pi:\tilde{C}\to C$ is the normalisation of $C$, then \[\RRES_x=\sum_{x'\in\pi^{-1}(x)}\RRES_{x'}.\] The sheaf of {\em regular differentials} $\Omega_{C/k}'$ is defined, for open $U\subseteq X$, by \begin{align*}\Omega_{C/k}'(U)=\{\omega\in\Omega_{K(C)/k}^1:&\RRES_x(f\omega)=0\mbox{ for all}\\ &\mbox{closed points } x\in U\mbox{ and all }f\in\roi_{C,x}\}.\end{align*} If $U$ contains no singular points of $C$, then $\Omega_{C/k}'|_U=\Omega_{U/k}^1$. By establishing a Riemann-Roch type result, it follows that $\Omega_{C/k}'$ is the dualizing sheaf of $C/k$. Analogously to the smooth case, one explicitly constructs the trace map \[\mbox{tr}_{C/k}:H^1(C,\Omega_{C/k}')\to k,\] and, as in \cite{Green1988}, uses it and ad\`elic spaces to prove duality. See \cite{Stohr1993} for more on the theory of regular differentials on curves.

\subsection{The case of a surface over a perfect field}\label{subsection_surface_over_perfect_field}
There is also a theory of residues on algebraic surfaces, developed by A. Parshin \cite{Parshin1983} \cite{Parshin2000}, the founder of the higher dimensional ad\`elic approach to algebraic geometry. Let $X$ be a connected, smooth, projective surface over a perfect field $k$. To each closed point $x\in X$ and curve $y\subset X$ containing $x$, he defined a two-dimensional residue map \[\RRES_{x,y}:\Omega_{K(X)/k}^2\to k\] and proved the reciprocity laws both around a point \[\sum_{\substack{y\subset X\\y\ni x}}\RRES_{x,y}\omega=0\] (for fixed $x\in X_0$ and $\omega\in\Omega_{K(X)/k}^2$) and along a curve \[\sum_{\substack{x\in X_0\\x\in y}}\RRES_{x,y}\omega=0\] (for fixed $y\subset X$ and $\omega\in\Omega_{K(X)/k}^2$). By interpreting the \v{C}ech cohomology of $X$ ad\`elically and proceeding analogously to the case of a curve, these residue maps may be used to explicitly construct the trace map \[\mbox{tr}_{X/k}:H^2(X,\Omega_{X/k}^2)\to k\] and, using two-dimensional ad\`elic spaces, prove duality.

D. Osipov \cite{Osipov2000} considers the algebraic analogue of our setting, with a smooth, projective surface $X$ over a perfect field $k$ and a projective morphism $f:X\to S$ to a smooth curve. To each closed point $x\in X$ and curve $y\subset X$ containing $x$, he constructs a `direct image map' \[f_{\ast}^{x,y}:\Omega_{K(X)/k}^2\to\Omega_{K(S)_s/k}^1,\] where $s=f(x)$ and $K(S)_s$ is the $s$-adic completion of $K(S)$. He establishes the reciprocity law around a point, analogous to our theorem \ref{theorem_reciprocity_around_a_point}, and the reciprocity law along a fibre, our remark \ref{remark_reciprocity_along_a_vertical_curve}. He uses the $(f_{\ast}^{x,y})_{x,y}$ to construct $f_{\ast}:H^2(X,\Omega_{X/k}^2)\to H^1(S,\Omega_{S/k}^1)$, which he proves is the trace map.

Osipov then considers multiplicative theory. Let $K_2(X)$ denote the sheafification of $X\supseteq U\mapsto K_2(\roi_X(U))$; then $H^2(X,K_2(X))\cong\mbox{CH}_2(X)$ by \cite{Bloch1974a}. Osipov defines, for each $x\in y\subset X$, homomorphisms \[f_{\ast}(\,,\,)_{x,y}:K_2(K(X))\to \mult{K(S)}_s,\] and establishes the reciprocity laws around a point and along a fibre. At least when $\mbox{char }k=0$, these are then used to construct a map \[\mbox{CH}^2(X)=H^2(X,K_2(X))\to H^1(C,\mult{\roi}_C)=\mbox{Pic}(C),\] which is proved to be the usual push-forward of cycles \cite[\S1]{Fulton1998}.

\subsection{Higher dimensions}
The theory of residues for surfaces was extended to higher dimensional varieties by V.~G.~Lomadze \cite{Lomadze1981}. Let $X$ be a $d$-dimensional, integral variety over a field $k$. A {\em chain}, or {\em flag} $\xi$ of length $n\ge 0$ is a sequence $\langle x_0,\dots,x_n\rangle$ of irreducible, closed subvarieties of $X$ such that $x_{i-1}\subseteq x_i$ for $i=1,\dots,n$; the chain is {\em complete} if and only if $\dim x_i=i$ for all $i$. For example, if $X$ is a surface, then a complete chain has the form $\langle x\in y\subset X\rangle$ where $y$ is a curve containing a closed point $x$. To each complete flag $\xi$ Lomadze associates a residue map $\RRES_{\xi}:\Omega_{K(X)/k}^d\to k$; he proves the reciprocity law \[\sum_{x_i}\RRES_{\xi}\omega=0\] for $\omega\in\Omega_{K(X)/k}^d$. Here we have fixed a flag $x_0\subset\dots\subset x_{i-1}\subset x_{i+1}\subset\dots\subset x_n$ (with $\dim x_i=i$ for each $i$) and vary the sum over all $i$-dimensional integral subvarieties $x_i$ sitting between $x_{i-1}$ and $x_{i+1}$ (if $i=0$ then we must assume $X$ is projective). 

Lomadze also develops a higher dimensional relative theory, analogous to Osipov's study of a surface over a curve.

\subsection{Explicit Grothendieck duality}\label{subsection_explicit_grothendieck_duality}
Given a proper morphism $\pi:X\to \Spec A$ of fibre dimension $\le r$ between Noetherian schemes, with the base affine (the only case we will encounter), Grothendieck duality \cite{Grothendieck1960} assures us of the existence of a quasi-coherent sheaf $\omegaa_\pi$, called the ($r$-)dualizing sheaf, together with a homomorphism $\mbox{tr}_\pi:H^r(X,\omegaa_\pi)\to A$ with the following property: for any quasi-coherent sheaf $\cal{F}$ on $X$, the natural pairing \[\Hom_{\roi_X}(\cal{F},\omegaa_\pi)\times H^r(X,\cal{F})\To H^r(X,\omegaa_\pi)\stackrel{\mbox{tr}_\pi}{\To} A\] induces an isomorphism of $A$-modules \[\Hom_{\roi_X}(\cal{F},\omegaa_\pi)\isoto \Hom_A(H^r(X,\cal{F}),A).\] This is an elementary form of Grothendieck duality, which is typically now understood as a statement about derived categories thanks to \cite{Hartshorne1966}, while \cite{Altman1970} and \cite{Kleiman1980} provide expositions closer in spirit to the statement we have given.

It is an interesting problem whether Grothendieck duality can be made more explicit. The guiding example is that of a curve over a field which we discussed above, where the trace map $\mbox{tr}_\pi$ may be constructed via residues. The duality theorem is even equivalent to Poisson summation on the ring of ad\`eles of the curve; the simplest exposition of duality is probably that of \cite{Moreno1991}. Using the Parshin-Lomadze theory of residues, A. Yekutieli \cite{Yekutieli1992} has explicitly constructed the Grothendieck residue complex of an arbitrary reduced scheme of finite type over a field. There is additional work on explicit duality by R.~H\"ubl and E.~Kunz \cite{Hubl1990} \cite{Hubl1990a}, and R.~H\"ubl and P.~Sastry \cite{Hubl1993}. The recent book by E.~Kunz \cite{Kunz2008} gives a complete exposition of duality for projective algebraic varieties using residue maps on local cohomology groups.

There are very close analogies between certain constructions and results of this paper and those of \cite{Yekutieli1992}; indeed, in the introduction to \cite{Yekutieli1992}, Yekutieli raises the problem of extending his results to schemes over a discrete valuation ring or $\Spec\bb{Z}$, and our results provide exactly that. To make this clearer, we now provide a summary of the relevant results of \cite{Yekutieli1992}. Let $X$ be an integral variety of dimension $d$ over a perfect field $k$. To each complete chain $\xi$ on $X$, Lomadze's theory of residues provides a natural residue map $\RRES_{\xi}:\Omega_{K(X)/k}^d\to k$. Given a codimension one irreducible, closed subvariety $y\subset X$, i.e. a prime divisor, Yekutieli defines, in \cite[Definition 4.2.3]{Yekutieli1992}, a form $\omega\in\Omega_{K(X)/k}^d$ to be {\em holomorphic along $y$} if and only if \[\RRES_{\xi}(f\omega)=0\] for all $f\in\roi_{X,y}$ and all complete chains $\xi$ of the form $\xi=\langle\dots,y,X\rangle$.

Having constructed the dualizing complex $\cal{K}_X^{\bullet}$, Yekutieli introduces $\tilde{\omegaa}_X:=H^{-n}\cal{K}_X^{\bullet}$, which is naturally contained inside the constant sheaf $\Omega_{K(X)/k}^d$. He proves \cite[Theorem 4.4.16]{Yekutieli1992} the analogue of our main theorem \ref{main_theorem}, namely that for open $U\subseteq X$, \begin{align*}\tilde{\omegaa}_X(U)=\{\omega\in\Omega_{K(X)/k}^d:&\,\omega\mbox{ is holomorphic along $y$,}\\ &\mbox{for all codimension one $y$ which meet U}\}.\end{align*} Furthermore, the idea of proof is identical, following our lemmas in section \ref{section_explicit_dualizing_sheaf_of_arith_surface}: he reduces the problem to the smooth case using functoriality of $\tilde{\omegaa}_X$ with respect to finite morphisms and trace maps, and then in the smooth case proves that $\tilde{\omegaa}_X=\Omega_{X/k}^d$. From this he concludes that $\tilde{\omega}_X$ is the sheaf of regular differential forms in the sense of E.~Kunz \cite{Kunz1988}.

Assuming that $X$ is generically smooth, then the sheaf of regular differentials is the dualizing sheaf, and the Grothendieck trace map may be constructed using residues, the reciprocity law, and ad\`elic spaces, in a similar way to the case of a curve or surface discussed above. See \cite{Yekutieli1992} and others of Yekutieli's papers, e.g. \cite{Hubl1996} \cite{Yekutieli1995}.

\subsection{Ad\`elic analysis}\label{subsection_adelic_analysis}
This work has many connections to I.~Fesenko's programme of two-dimensional ad\`elic analysis \cite{Fesenko2003} \cite{Fesenko2006} \cite{Fesenko2008} \cite{Fesenko2008a} \cite {Morrow2008} \cite{Morrow2008a} \cite{Morrow2008b}, and is part of the author's attempt to understand the connection between ad\`elic analysis and more familiar methods in algebraic geometry.

Two-dimensional ad\`elic analysis aims to generalise the current rich theories of topology, measure, and harmonic analysis which exist for local fields, by which mathematicians study curves and number fields, to dimension two. In particular, Fesenko generalises the Tate-Iwasawa \cite{Iwasawa1992} \cite{Tate1967} method of studying the zeta function of a global field to dimension two, giving a new approach to the study of the $L$-function of an elliptic curve over a global field. The author hopes that the reader is satisfied to hear only the most immediate relations between this fascinating subject and the current paper.

Let $E$ be an elliptic curve over a number field $K$, with function field $F=K(E)$, and let $\cal{E}$ be a regular, projective model of $E$ over the ring of integers $\roi_K$. Then $\cal{E}$ satisfies the assumptions which we impose on our arithmetic surfaces in this paper. Let $\psi=\otimes_{s\in S_0}\psi_s:\bb{A}_K\to S^1$ be an additive character on the ad\`ele group of $K$, and let $\omega\in\Omega_{F/K}^1$ be a fixed, non-zero differential form. For $x\in y\subset\cal{E}$ a point contained in a curve as usual, with $x$ sitting over $s\in S$, introduce an additive character \[\psi_{x,y}:F_{x,y}\to S^1,\quad a\mapsto \psi_s(\RRES_{x,y}(a\omega)),\] where $\RRES_{x,y}$ is the relative residue map which we will construct in section \ref{section_reciprocity_for_surfaces}. If $x$ is a fixed point, then our reciprocity law will imply \[\sum_{\substack{y\subset X\\y\ni x}}\psi_{x,y}(a)=0\] for any $a\in F$. 

Moreover, suppose that $\psi$ is trivial on global elements and that $y$ is a fixed horizontal curve; then Fesenko also proves \cite[\S27 Proposition]{Fesenko2008a} \[\sum_{\substack{x\in X_0\\x\in y\cup\{\sub{arch}\}}}\psi_{x,y}(a)=0.\] We are deliberately vague here. Let us just say that we must adjoin archimedean points to $S$ and $y$, consider two-dimensional archimedean local fields such as $\bb{R}((t))$, and define suitable additive characters at these places; once these have been suitably introduced, this reciprocity law follows from ad\`elic reciprocity for the number field $k(y)$.

\subsection{Future work}\label{subsection_arithmetic_surfaces_future_work}
The author is thinking about several topics related to this paper which may interest the reader. Let $\pi:X\to \roi_K$ be an arithmetic surface.

\subsubsection*{Grothendieck duality}
In the existing work on explicit Grothendieck duality on algebraic varieties using residues, there are three key steps. Firstly one must define suitable local residue maps, either on spaces of differential forms or on local cohomology groups. Secondly, the local residue maps are used to define the dualizing sheaf, and finally the local residue maps must be patched together in some fashion to define Grothendieck's trace map on the cohomology of the dualizing sheaf.

In this paper the first two steps are carried out for arithmetic surfaces. It should be possible to use our residue maps $(\RRES_{x,y})_{x,y}$ to construct the trace map \[\mbox{tr}_\pi:H^1(X,\omegaa_\pi)\to\roi_K,\] and give an explicit ad\`elic proof of Grothendieck duality, similar to the existing work for algebraic varieties described above. This should follow relatively easily from the contents of this paper, and the author hopes to publish it at a later time\footnote{See \cite{Morrow2010}.}.

\subsubsection*{Horizontal reciprocity}
If $y$ is horizontal then reciprocity law along $y$ does not make sense naively (in contrast with a vertical curve: see remark \ref{remark_reciprocity_along_a_vertical_curve}), since the residues $\RRES_{x,y}\omega$ belong to different fields as $x$ varies across $y$. Of course, this is the familiar problem that $\Spec\roi_K$ is not a relative curve. As explained in the discussion of Fesenko's work above, this is fixed by taking into account the archimedean data. Such results seem to live outside the realm of algebraic geometry, and need to be better understood.

\subsubsection*{Two-dimensional Poisson summation}
Perhaps it is possible to find a global duality result on $X$ which incorporates not only Grothendieck duality of $X$ relative to $S$, but also the arithmetic duality on the base i.e. Poisson summation. Such a duality would necessarily incorporate archimedean data and perhaps be most easily expressed ad\`elically. Perhaps it already exists, in the form of Fesenko's two-dimensional theta formula \cite[\S3.6]{Fesenko2008a}.

\subsubsection*{Multiplicative theory}
We have focused on additive theory, but as we mentioned while discussing Osipov's work, there are natural multiplicative analogues. In fact, the `multiplicative residue map' for mixed characteristic two-dimensional local fields has been defined by K.~Kato \cite{Kato1983}. Fesenko's work includes an ad\`elic interpretation of the conductors of the special fibres of $\cal{E}$, but only under the assumption that the reduced part of each fibre is semi-stable \cite[\S40, Remark 2]{Fesenko2008a}; similar results surely hold in greater generality and are related to `conductor $=$ discriminant' formulae \cite{Kato2004} \cite{Liu2000} \cite{Saito1988}.

Moreover, Fesenko's two-dimensional theta formula \cite[3.6]{Fesenko2008a} is an ad\`elic duality which takes into account the interplay between the additive and multiplicative structures. It is important to understand better its geometric interpretation, at least in the case of an algebraic surface.

Perhaps it is also possible to study vanishing cycles \cite{Saito1987} using similar techniques.

\subsection{Notation}
If $A$ is a (always commutative) ring, then we write $\frak{p}\lhd A$ to denote that $\frak{p}$ is an ideal of $A$; this notation seems to be common to those educated in Oxford, and less familiar to others. We write $\frak{p}\lhd\!^1A$ to indicate that the height of $\frak{p}$ is $1$. If $\frak{p}$ is prime, then $k(\frak{p})=\Frac A/\frak{p}$ is the residue field at $\frak{p}$. If $A$ is a local ring, then the maximal ideal is $\frak{m}_A$.

If $F$ is a complete discrete valuation field, then its ring of integers is $\roi_F$, with maximal ideal $\frak{p}_F$. The residue field $k(\frak{p}_F)$ will be denoted $\res{F}$; this notation seems to be common among those affected by the Russian school of arithmetic geometry. Discrete valuations are denoted $\nu$, usually with an appropriate subscript to avoid confusion.

If $A$ is a $B$-algebra, the the space of relative Kahler differentials is $\Omega_{A/B}=\Omega_{A/B}^1$, and the universal differential is denoted $d:A\to\Omega_{A/B}$. For $a\in A$, we often write $da$ in place of $d(a)$.

Injective maps are often denoted by $\into$, and surjective maps by $\onto$.

\subsection{Acknowledgement}
I am grateful to my supervisor I. Fesenko for suggesting that I think about arithmetic ad\`elic duality. He and L. Xiao gave me some invaluable help on appendix \ref{section_finite_morphisms_etc}, the contents of which were essential in an earlier version of the paper. A key idea in the proof of lemma \ref{lemma_normalization_result}, namely the use of the prime avoidance lemma, was kindly shown to me by David E Speyer via the online forum {\tt mathoverflow.net}, and any subsequent errors in the proof are due entirely to myself. Finally, I would like to thank the anonymous referee for reading the paper with great diligence. 

Parts of this research were done while visiting the IH\'ES (November 2008) and Harvard University (Spring 2009). These extended trips would not have been possible without funding provided through the London Mathematical Society in the form of the Cecil King Travel Scholarship, and I would like to thank both the Cecil King Memorial Foundation for its generosity and the institutes for their hospitality.

\section{Local relative residues}\label{section_local_relative_residues}
Here we develop a theory of residues of differential forms on two-dimensional local fields. Recall that a two-dimensional local field is a complete discrete valuation field $F$ whose residue field $\res{F}$ is a (non-archimedean, in this paper) local field. We will be interested in such fields $F$ of characteristic zero; when the local field $\res{F}$ also has characteristic zero then we say that $F$ has {\em equal characteristic zero}; when $\res{F}$ has finite characteristic, then $F$ is said to be of {\em mixed characteristic}.

\subsection{Continuous differential forms}
We begin by explaining how to construct suitable spaces of `continuous' differential forms.

For any Noetherian, local ring $A$ and $A$-module $N$, we will denote by $N^\sub{sep}$ the maximal Hausdorff (=separated) quotient for the $\frak{m}_A$-adic topology, i.e. \[N^\sub{sep}=N\left/\bigcap_{n=1}^{\infty}\frak{m}_A^nN\right..\]

\begin{remark}\label{remark_separated_quotients_under_extensions}
Suppose that $A/B$ is a finite extension of Noetherian, local domains. Then $\frak{m}_A\cap B=\frak{m}_B$. Also, the fibre $A\otimes_B k(\frak{m}_B)$ is a finite dimensional $k(\frak{m}_B)$-vector space, and is therefore Artinian; hence $\frak{m}_BA$ contains $\frak{m}_A^n$ for $n\gg 0$. So for any $B$-module $N$, \[N^\sub{sep}\otimes_B A=(N\otimes_B A)^\sub{sep}.\]
\end{remark}

\begin{lemma}\label{lemma_differential_forms_of_tower}
Let $A/B$ be a finite extension of Noetherian, local domains, which are $R$-algebras, where $R$ is a Noetherian domain. Assume that $\Omega_{B/R}^\sub{sep}$ is a free $B$-module, and that $\Frac A/\Frac B$ is a separable extension. Then there is an exact sequence \[0\to\Omega_{B/R}^\sub{sep}\otimes_B A\to\Omega_{A/R}^\sub{sep}\to\Omega_{A/B}\to 0\] of $A$-modules.
\end{lemma}
\begin{proof}
The standard exact sequence of differential forms is \[\Omega_{B/R}\otimes_B A\to\Omega_{A/R}\to\Omega_{A/B}\to 0.\] Since $A$ is a finite $B$-module, the space of differentials $\Omega_{A/B}$ is a finitely generated, torsion $A$-module. Apply $^\sub{sep}$ to the sequence to obtain, using remark \ref{remark_separated_quotients_under_extensions}, \[\Omega_{B/R}^\sub{sep}\otimes_B A\stackrel{j}{\to}\Omega_{A/R}^\sub{sep}\to\Omega_{A/B}\to 0,\] which is exact. It remains to prove that $j$ is injective.

Let $F$, $M$, $K$ be the fields of fractions of $A$, $B$, $R$ respectively, and let $\omega\in \Omega_{B/R}^\sub{sep}$ be an element of some chosen $B$-basis for this free module. Let $D_{\omega}:\Omega_{B/R}^\sub{sep}\to B$ send $\omega$ to $1$ and vanish on all other elements of the chosen basis. This homomorphism extends first to an $M$-linear map $D_M:\Omega_{M/K}\to M$, and then to an $F$-linear map $D_F:\Omega_{F/K}\to F$; this follows from the identifications $\Omega_{B/R}\otimes_BM \cong\Omega_{M/K}$ and $\Omega_{M/K}\otimes_M F\cong\Omega_{F/K}$. Finally, it induces an $R$-linear derivation $D:A\to F$ by $D(a)=D_F(d(a))$, where $d:F\to\Omega_{F/K}$ is the universal derivation.

Let $N\subseteq F$ be the $A$-module spanned by $D(a)$, for $a\in A$. This is a finitely generated $A$-module, for if $a_1,\dots,a_n$ generate $A$ as a $B$-module, then $N$ is contained in the $A$-module spanned by $a_1,\dots,a_n,D(a_1),\dots,D(a_n)$. Thus the non-zero homomorphism $\tilde{D}:\Omega_{A/R}\to N$ induced by $D$ factors through $\Omega_{A/R}^\sub{sep}$ (by Nakayama's lemma). Furthermore, $\tilde{D}$ sends $j(\omega)\in\Omega_{A/R}^\sub{sep}$ to $1$ and vanishes on the images under $j$ of the other basis elements. It follows that $j$ is injective.
\end{proof}

\begin{remark}
Whether $\Omega_{B/R}^\sub{sep}$ is free is closely related to whether $B$ is a formally smooth algebra over $R$; see \cite[Th\'eor\`eme 20.5.7]{EGA_IV_I}. M.~Kurihara uses such relations more systematically in his study of complete discrete valuation fields of mixed characteristic \cite{Kurihara1987}.
\end{remark}

\begin{remark}\label{remark_completion_of_fin_gen_algebra}
Suppose that $R$ is a Noetherian ring and $A$ is a finitely generated $R$-algebra. Let $\frak{p}\lhd A$ be a prime ideal. Then $\Omega_{A_{\frak{p}}/R}=\Omega_{A/R} \otimes_AA_{\frak{p}}$ is a finitely generated $A_{\frak{p}}$-module, and the natural map $\Omega_{A_{\frak{p}}/R}\otimes_{A_{\frak{p}}}\hat{A_{\frak{p}}}\to\Omega_{\hat{A_{\frak{p}}}/R}$ gives rise to an isomorphism \[\Omega_{A_{\frak{p}}/R}\otimes_{A_{\frak{p}}}\hat{A_{\frak{p}}}\cong\lim_{\substack{\longleftarrow\\n}}\Omega_{\hat{A_{\frak{p}}}/R}/\frak{p}^n\Omega_{\hat{A_{\frak{p}}}/R}=\hat{\Omega_{\hat{A_{\frak{p}}}/R}}\] (see e.g. \cite[exercise 6.1.13]{Liu2002}).

Therefore $\Omega_{\hat{A_{\frak{p}}}/R}^\sub{sep}$ is a finitely generated $\hat{A_{\frak{p}}}$-module (since it embeds into $\hat{\Omega_{\hat{A_{\frak{p}}}/R}}$), and it is therefore complete; so the embedding $\smash{\Omega_{\hat{A_{\frak{p}}}/R}^\sub{sep}\into\hat{\Omega_{\hat{A_{\frak{p}}}/R}}}$ is actually an isomorphism. Thus we have a natural isomorphism \[\Omega_{A/R}\otimes_A\hat{A_{\frak{p}}}\cong\Omega_{\hat{A_{\frak{p}}}/R}^\sub{sep}.\]

We will occasionally give explicit proofs of results which could otherwise be deduced from this remark.
\end{remark}

\begin{definition}
Let $F$ be a complete discrete valuation field, and let $K$ be a subfield of $F$ such that $\Frac(K\cap\roi_F)=K$. The space of {\em continuous relative differentials} is \[\Omega_{F/K}^\sub{cts}:=\Omega_{\roi_F/K\cap\roi_F}^\sub{sep}\otimes_{\roi_F}F.\] It is vector space over $F$ and there is a natural surjection $\Omega_{F/K}\onto\Omega_{F/K}^\sub{cts}$.
\end{definition}

\begin{remark}
Suppose that $F$, $K$ are as in the previous definition, and that $F'$ is a finite, separable extension of $F$. Using remark \ref{remark_separated_quotients_under_extensions}, one shows $\Omega_{F'/K}^\sub{cts}=\Omega_{F/K}^\sub{cts}\otimes_FF'$, and therefore there is a well-defined trace map $\Tr_{F'/F}:\Omega_{F'/K}^\sub{cts}\to\Omega_{F/K}^\sub{cts}.$
\end{remark}

\subsection{Equal characteristic}\label{subsection_equal_characteristic}
We begin with residues in the equal characteristic case; this material is well-known (see e.g. \cite{Serre1988}) so we are brief. Let $F$ be a two-dimensional local field of equal characteristic zero. We assume that an embedding of a local field $K$ (necessarily of characteristic zero) into $F$ is given; such an embedding will be natural in our applications. The valuation $\nu_F|_K$ must be trivial, for else it would be a multiple of $\nu_K$ (a complete discrete valuation field has a unique normalised discrete valuation) which would imply $\res{K}\into\res{F}$, contradicting our hypothesis on the characteristic of $\res{F}$; so $K\subseteq\roi_F$ and $K\into\res{F}$, making $\res{F}$ into a finite extension of $K$.

\begin{lemma}\label{lemma_unique_coefficient_field}
$F$ has a unique coefficient field which contains $K$.
\end{lemma}
\begin{proof}
Set $n=|\res{F}:K|$. Suppose first that $K'/K$ is any finite subextension of $F/K$. Then $K'\subseteq\roi_F$ and so the residue map restricts to a $K$-linear injection $K'\into\res{F}$, proving that $|K':K|\le n$. This establishes that $K$ has at most one extension of degree $n$ inside $F$ (for if there were two extensions then we could take their composite), and that if such an extension exists then it is the desired coefficient field (for then the residue map $K'\into\res{F}$ must be an isomorphism).

Since $K$ is perfect, apply Hensel's lemma to lift to $\roi_F$ a generator for $\res{F}/K$; the subextension of $F/K$ generated by this element has degree $n$, completing the proof.
\end{proof}

This unique coefficient field will be denoted $k_F$; it depends on the image of the embedding $K\into\roi_F$, though the notation does not reflect that. $k_F$ is a finite extension of $K$; moreover, it is simply the algebraic closure of $K$ inside $F$. When the local field $K\subseteq F$ has been fixed, we will refer to $k_F$ as {\em the} coefficient field of $F$ ({\em with respect to $K$}, if we want to be more precise). Standard structure theory implies that choosing a uniformiser $t\in F$ induces a $k_F$-isomorphism $F\cong k_F((t))$.

\begin{lemma}\label{lemma_differential_forms_of_K[[t]]}
$\Omega_{\roi_F/\roi_K}^\sub{sep}$ is a free $\roi_F$-module of rank $1$, with basis $dt$, where $t$ is any uniformiser of $F$. Hence $\Omega_{F/K}^\sub{cts}$ is a one-dimensional vector space over $F$ with basis $dt$.
\end{lemma}
\begin{proof}
Any derivation on $\roi_F$ which vanishes on $\roi_K$ also vanishes on $K$, and it even vanishes on $k_F$ since $k_F/K$ is a finite, separable extension. Hence $\Omega_{\roi_F/\roi_K}=\Omega_{\roi_F/K}=\Omega_{\roi_F/k_F}$.

Fix a uniformiser $t\in F$, to induce a $k_F$-isomorphism $\roi_F\cong k_F[[t]]$. This allows us to define a $k_F$-linear derivation $\frac{d}{dt}:\roi_F\to\roi_F,\,\sum_ia_it^i\mapsto\sum_iia_it^{i-1}$. For any $f\in\roi_F$ and $n\ge 0$, we may write $f=\sum_{i=0}^na_it^i+gt^{n+1}$, with $a_0,\dots,a_n\in k_F$ and $g\in\roi_F$; let $d:\roi_F\to\Omega_{\roi_F/\roi_K}$ be the universal derivation and apply $d$ to obtain \[df=\sum_{i=0}^na_iit^{i-1}\,dt+g(n+1)t^ndt+t^{n+1}dg.\] It follows that $df-\frac{df}{dt}dt\in\bigcap_{n=1}^{\infty}t^n\Omega_{\roi_F/k_F}$. Taking the separated quotient shows that $dt$ generates $\Omega_{\roi_F/k_F}^\sub{sep}$; the existence of the derivation $\frac{d}{dt}$ implies that $dt$ is not torsion.
\end{proof}

The {\em residue map of $F$, relative to $K$} is defined by \[\RES_F:\Omega_{F/K}^\sub{cts}\to k_F,\quad \omega=fdt\mapsto\mbox{coeft}_{t^{-1}}(f),\] where the notation means that we take the coefficient of $t^{-1}$ in the expansion of $f$. Implicit in the definition is the choice of a $k_F$-isomorphism $F\cong k_F((t))$.

It is well-known that the residue map does not depend on the choice of uniformiser $t$. Since the proof is straightforward in residue characteristic zero, we recall it. Any other uniformiser $T$ has the form $T=\sum_{i=1}^\infty a_it^i$ with $a_i\in k_F$ and $a_1\neq 0$; for $j\in\bb{Z}\setminus{\{-1\}}$, we have \[\mbox{coeft}_{t^{-1}}\left(T^j\frac{dT}{dt}\right)=\mbox{coeft}_{t^{-1}}\left(\frac{1}{j+1}\frac{dT^{j+1}}{dt}\right)=0.\] When $j=-1$, we instead calculate as follows: \[\mbox{coeft}_{t^{-1}}\left(T^{-1}\frac{dT}{dt}\right)=\mbox{coeft}_{t^{-1}}((a_1^{-1}t^{-1}-a_1^{-2}a_2+\dots)(a_1+2a_2t+\dots))=1.\] Finally, since the residue is continuous with respect to the discrete valuation topology on $\Omega_{F/K}^\sub{cts}=Fdt$ and the discrete topology on $k_F$, we have \[\mbox{coeft}_{t^{-1}}\left(\sum_{j\gg-\infty}b_jT^j\frac{dT}{dt}\right)=b_{-1},\] and it follows that the residue map may also be defined with respect to the isomorphism $F\cong k_F((T))$.

Now we recall functoriality of the residue map. Note that if $F'$ is a finite extension of $F$, then there is a corresponding finite extension $k_{F'}/k_F$ of the coefficient fields.

\begin{proposition}\label{proposition_functoriality_in_equal_characteristic_case}
Let $F'$ be a finite extension of $F$. Then the following diagram commutes:
\[\begin{CD}
\Omega_{F'/K}^\sub{cts} &@>\RES_{F'}>>&k_{F'}\\
@V\Tr_{F'/F}VV && @VV\Tr_{k_{F'}/k_F}V\\
\Omega_{F/K}^\sub{cts} &@>\RES_F>>&k_F
\end{CD}\]
\end{proposition}	
\begin{proof}
This is another well-known result, whose proof we give since it is easy in the characteristic zero case. It suffices to consider two separate cases: when $F'/F$ is unramified, and when $F'/F$ is totally ramified (as extensions of complete discrete valuation fields).

In the unramified case, $|k_{F'}:k_F|=|F':F|$ and we may choose compatible isomorphisms $F\cong k_F((t))$, $F'\cong k_{F'}((t))$; the result easily follows in this case.

In the totally ramified case, $F'/F$ is only tamely ramified, $k_{F'}=k_F$, and we may choose compatible isomorphisms $F\cong k_F((t))$, $F'\cong k_{F'}((T))$, where $T^e=t$. We may now follow the argument of \cite[II.13]{Serre1988}.
\end{proof}

\subsection{Mixed characteristic}
Now we introduce relative residue maps for two-dimensional local fields of mixed characteristic. We take a local, explicit approach, with possible future applications to higher local class field theory and ramification theory in mind. This residue map is implicitly used by Fesenko \cite[\S3]{Fesenko2003} to define additive characters in his two-dimensional harmonic analysis.

\subsubsection{Two-dimensional local fields of mixed characteristic}
We begin with a review of this class of fields.

\begin{example}\label{example_double_curly_field}
Let $K$ be a complete discrete valuation field. Let $K\dblecurly{t}$ be the following collection of formal series:
\begin{align*}K\dblecurly{t}
	=\left\{\sum_{i=-\infty}^{\infty}a_it^i :\right.&\, a_i\in K\mbox{ for all }i,\,\inf_i\nu_K(a_i)>-\infty,\\
	&\left.\phantom{\sum_{i=-\infty}^{\infty}}\mbox{and }a_i\to 0\mbox{ as } i\to -\infty\right\}.
\end{align*}
Define addition, multiplication, and a discrete valuation by
\begin{align*}
\sum_{i=-\infty}^{\infty}a_it^i + \sum_{j=-\infty}^{\infty}a_jt^j&=\sum_{i=-\infty}^{\infty}(a_i+b_i)t^i\\
\sum_{i=-\infty}^{\infty}a_it^i\cdot\sum_{j=-\infty}^{\infty}a_jt^j&=\sum_{i=-\infty}^{\infty}\left(\sum_{r=-\infty}^{\infty}a_rb_{i-r}\right)t^i\\
\nu\left(\sum_{i=-\infty}^{\infty}a_it^i\right)&=\inf_i\nu_K(a_i)
\end{align*}
Note that there is nothing formal about the sum over $r$ in the definition of multiplication; rather it is a convergent double series in the complete discrete valuation field $K$. These operations are well-defined, make $K\dblecurly{t}$ into a field, and $\nu$ is a discrete valuation under which $K\dblecurly{t}$ is complete. Note that $K\dblecurly{t}$ is an extension of $K$, and that $\nu|_K=\nu_K$, i.e. $e(K\dblecurly{t}/K)=1$.

The ring of integers of $K\dblecurly{t}$ and its maximal ideal are given by
\begin{align*}
\roi_{K\dblecurly{t}}&=\left\{\sum_ia_it^i:a_i\in\roi_K\mbox{ for all }i\mbox{ and }a_i\to\infty\mbox{ as }i\to-\infty\right\},\\
\frak{p}_{K\dblecurly{t}}&=\left\{\sum_ia_it^i:a_i\in\frak{p}_K\mbox{ for all }i\mbox{ and }a_i\to\infty\mbox{ as }i\to-\infty\right\}.
\end{align*}
The surjective homomorphism \[\roi_{K\dblecurly{t}}\to\res{K}((t)),\quad\sum_ia_it^i\mapsto\sum_i\res{a}_it^i\] identifies the residue field of $K\dblecurly{t}$ with $\res{K}((t))$.

The alternative description of $K\dblecurly{t}$ is as follows. It is the completion of $\Frac(\roi_K[[t]])$ with respect to the discrete valuation associated to the height one prime ideal $\pi_K\roi_K[[t]]$.
\end{example}

We will be interested in the previous example when $K$ is a local field of characteristic $0$. In this case, $K\dblecurly{t}$ is a two-dimensional local field of mixed characteristic.

Now suppose $L$ is any two-dimensional local field of mixed characteristic of residue characteristic $p$. Then $L$ contains $\bb{Q}$, and the restriction of $\nu_L$ to $\bb{Q}$ is a valuation which is equivalent to $\nu_p$, since $\nu_L(p)>0$; since $L$ is complete, we may topologically close $\bb{Q}$ to see that $L$ contains a copy of $\bb{Q}_p$. It is not hard to see that this is the unique embedding of $\bb{Q}_p$ into $L$, and that $L/\bb{Q}_p$ is an (infinite) extension of discrete valuation fields. The corresponding extension of residue fields is $\res{L}/\bb{F}_p$, where $\res{L}$ is a local field of characteristic $p$.

The analogue of the coefficient field in the equal characteristic case is the following:

\begin{definition}
The {\em constant} subfield of $L$, denoted $k_L$, is the algebraic closure of $\bb{Q}_p$ inside $L$.
\end{definition}

\begin{lemma}
If $K$ is an arbitrary field then $K$ is relatively algebraically closed in $K((t))$. If $K$ is a complete discrete valuation field then $K$ is relatively algebraically closed in $K\dblecurly{t}$; so if $K$ is a local field of characteristic zero, then the constant subfield of $K\dblecurly{t}$ is $K$.
\end{lemma}
\begin{proof}
Suppose that there is an intermediate extension $K((t))\ge L\ge K$ with $L$ finite over $K$. Then each element of $L$ is integral over $K[[t]]$, hence belongs to $K[[t]]$. The residue map $K[[t]]\to K$ is non-zero on $L$, hence restricts to a $K$-algebra injection $L\into K$. This implies $L=K$.

Now suppose $K$ is a complete discrete valuation field and that we have an intermediate extension $K\dblecurly{t}\ge M\ge K$ with $M$ finite over $K$. Then $M$ is a complete discrete valuation field with $e(M/K)=1$, since $e(K\dblecurly{T}/K)=1$. Passing to the residue fields and applying the first part of the proof to $\res{K}((t))$ implies $f(M/K)=1$. Therefore $|M:K|=1$, as required.
\end{proof}

Let $L$ be a two-dimensional local field of mixed characteristic. The algebraic closure of $\bb{F}_p$ inside $\res{L}$ is finite over $\bb{F}_p$ (it is the coefficient subfield of $\res{L}$); so, if $k$ is any finite extension of $\bb{Q}_p$ inside $L$, then $f(\res{k}/\bb{F}_p)$ is bounded above. But also $e(k/\bb{Q}_p)<e(L/\bb{Q}_p)<\infty$ is bounded above. It follows that $k_L$ is a finite extension of $\bb{Q}_p$.

Thus the process of taking constant subfields canonically associates to any two-dimensional local field $L$ of mixed characteristic a finite extension $k_L$ of $\bb{Q}_p$.

\begin{lemma}\label{lemma_constant_extension_of_standard_field}
Suppose $K$ is a complete discrete valuation field and $\Omega/K$ is a field extension with subextensions $F,K'$ such that $K'/K$ is finite and separable, and $F$ is $K$-isomorphic to $K\dblecurly{T}$. Then the composite extension $FK'$ is $K$-isomorphic to $K'\dblecurly{T}$.
\end{lemma}
\begin{proof}
Let $K''$ be the Galois closure of $K'$ over $K$ (enlarging $\Omega$ if necessary); then the previous lemma implies that $K''\cap F=K$ and therefore the extensions $K'',F$ are linearly disjoint over $K$ (here it is essential that $K''/K$ is Galois). This implies that $FK''$ is $K$-isomorphic to $F\otimes_K K''$, which is easily seen to be $K$-isomorphic to $K''\dblecurly{T}$. The resulting isomorphism $\sigma:FK''\to K''\dblecurly{T}$ restricts to an isomorphism $FK'\to \sigma(K')\dblecurly{T}$, and this final field is isomorphic to $K'\dblecurly{T}$.
\end{proof}

\begin{lemma}\label{lemma_existence_of_standard_subfield}
Suppose $L$ is a two-dimensional local field of mixed characteristic. Then there is a two-dimensional local field $M$ contained inside $L$, such that $L/M$ is a finite extension and
\begin{enumerate}
\item $\res{M}=\res{L}$;
\item $k_M=k_L$;
\item $M$ is $k_M$-isomorphic to $k_M\dblecurly{T}$.
\end{enumerate}
\end{lemma}
\begin{proof}
The residue field of $L$ is a local field of characteristic $p$, and therefore there is an isomorphism $\res{L}\cong\bb{F}_q((t))$; using this we may define an embedding $\bb{F}_p((t))\into\res{L}$, such that $\res{L}/\bb{F}_p((t))$ is an unramified, separable extension. Since $\bb{Q}_p\dblecurly{t}$ is an absolutely unramified discrete valuation field with residue field $\bb{F}_p((t))$, a standard structure theorem of complete discrete valuation fields \cite[Proposition 5.6]{Fesenko2002} implies that there is an embedding of complete discrete valuation fields $j:\bb{Q}_p\dblecurly{t}\into L$ which lifts the chosen embedding of residue fields. Set $F=j(\bb{Q}_p\dblecurly{t})$, and note that $f(L/F)=|\res{L}:\bb{F}_p((t))|=\log_p(q)$ and $e(L/F)=\nu_L(p)<\infty$; so $L/F$ is a finite extension.

Now apply the previous lemma with $K=\bb{Q}_p$ and $K'=k_L$ to obtain $M=FK'\cong k_L\dblecurly{t}$. Moreover, Hensel's lemma implies that $L$, and therefore $k_L$, contains the $q-1$ roots of unity; so $\res{k}_L\res{F}=\bb{F}_q\cdot\bb{F}_p((t))=\res{L}$, and therefore $\res{M}=\res{L}$.
\end{proof}

We will frequently use arguments similar to those of the previous lemma in order to obtain suitable subfields of $L$.

\begin{definition}
A two-dimensional local field $L$ of mixed characteristic is said to be {\em standard} if and only if $e(L/k_L)=1$.
\end{definition}

The purpose of the definition is to provide a `co-ordinate'-free definition of the class of fields we have already considered:

\begin{corollary}\label{corollary_structure_of_standard_fields}
$L$ is standard if and only if there is a $k_L$-isomorphism $L\cong k_L\dblecurly{t}$. If $L$ is standard and $k'$ is a finite extension of $k_L$, then $Lk'$ is also standard, with constant subfield $k'$.
\end{corollary}
\begin{proof}
Since $e(k_L\dblecurly{t}/k_L)=1$, the field $L$ is standard if it is isomorphic to $k_L\dblecurly{t}$. Conversely, by the previous lemma, there is a standard subfield $M\le L$ with $k_M=k_L$ and $\res{M}=\res{L}$; then $e(M/k_M)=1$ and $e(L/k_L)=1$ (since we assumed $L$ was standard), so that $e(L/M)=1$ and therefore $L=M$.

The second claim follows from lemma \ref{lemma_constant_extension_of_standard_field}.
\end{proof}

\begin{remark}\label{remark_local_parameter_determines_isomorphism}
A {\em first local parameter} of a two-dimensional local field $L$ is an element $t\in \roi_L$ such that $\res{t}$ is a uniformiser for the local field $\res{L}$. For example, $t$ is a first local parameter of $K\dblecurly{t}$. More importantly, if $L$ is standard, then any isomorphism $k_L\dblecurly{t}\isoto L$ is determined by the image of $t$, and conversely, $t$ may be sent to any first local parameter of $L$. This follows from similar arguments to those found in lemma \ref{lemma_existence_of_standard_subfield} above and \ref{lemma_structure_of_differential_forms_standard_fields} below; see e.g. \cite[Proposition 5.6]{Fesenko2002} and \cite{Madunts1995}. We will abuse notation in a standard way, by choosing a first local parameter $t\in L$ and then identifying $L$ with $k_L\dblecurly{t}$.
\end{remark}

\subsubsection{The residue map for standard fields.}
Here we define a residue map for standard two-dimensional fields and investigate its main properties. As in the equal characteristic case, we work in the relative situation, with a fixed standard two-dimensional local field $L$ of mixed characteristic and a chosen (one-dimensional) local field $K\le L$. It follows that $K$ is intermediate between $\bb{Q}_p$ and the constant subfield $k_L$.

We start by studying spaces of differential forms. Note that if we choose a first local parameter $t\in L$ to induce an isomorphism $L\cong k_L\dblecurly{t}$, then there is a well-defined $k_L$-linear derivative $\frac{d}{dt}:L\to L,\,\sum_ia_it^i\mapsto\sum_iia_it^{i-1}$.

\begin{lemma}\label{lemma_structure_of_differential_forms_standard_fields}
Let $t$ be any first local parameter of $L$. Then $\Omega_{\roi_L/\roi_K}^\sub{sep}$ decomposes as a direct sum \[\Omega_{\roi_L/\roi_K}^\sub{sep}=\roi_L dt\oplus\Tors(\Omega_{\roi_L/\roi_K}^\sub{sep})\] with $\roi_L dt$ free, and $\Tors(\Omega_{\roi_L/\roi_K}^\sub{sep})\cong\Omega_{\roi_{k_L}/\roi_K}\otimes_{\roi_{k_L}}\roi_L$. Hence $\Omega_{L/K}^\sub{cts}$ is a one-dimensional vector space over $L$ with basis $dt$.
\end{lemma}
\begin{proof}
First suppose that $K=k_L$ is the constant subfield of $L$. Then we claim that for any $f\in\roi_L$, one has $df=\frac{df}{dt}dt$ in $\Omega_{\roi_L/\roi_K}^\sub{sep}$.

Standard theory of complete discrete valuation fields (see e.g. \cite{Madunts1995}) implies that there exists a map $H:\res{L}\to\mult{\roi}_L\cup\{0\}$ with the following properties:
\begin{enumerate}
\item $H$ is a lifting, i.e. $\res{H}(a)=a$ for all $a\in\res{L}$;
\item $H(\res{t})=t$;
\item for any $a_0,\dots,a_{p-1}\in \res{L}$, one has $H(\sum_{i=0}^{p-1}a_i^p\res{t}^i)=\sum_{i=0}^{p-1}H(a_i)^pt^i$.
\end{enumerate}
The final condition replaces the Teichmuller identity $H(a^p)=a^p$ which ones sees in the perfect residue field case. We will first prove our claim for elements of the form $f=H(a)$, for $a\in\res{L}$. Indeed, for any $n>0$, we expand $a$ using the $p$-basis $\res{t}$ to write \[a=\sum_{i=0}^{p^n-1}a_i^{p^n}\res{t}^i\] for some $a_0,\dots,a_{p^{n-1}}\in\res{L}$. Lifting, and using the Teichmuller property (iii) of $H$ $n$-times, obtains \[f=\sum_{i=0}^{p^n-1}H(a_i)^{p^n}t^i.\] Now apply the universal derivative to reveal that \[df=\sum_{i=0}^{p^n-1}H(a_i)^{p^n}it^{i-1}dt+p^nH(a_i)^{p^n-1}t^id(H(a_i)).\] We may apply $\frac{d}{dt}$ in a similar way, and it follows that $df-\frac{df}{dt}dt\in p^n\Omega_{\roi_L/\roi_K}$. Letting $n\to\infty$ gives us $df=\frac{df}{dt}dt$ in $\Omega_{\roi_L/\roi_K}^\sub{sep}$.

Now suppose that $f\in\roi_L$ is not necessarily in the image of $H$. For any $n$, we may expand $f$ as a sum \[f=\sum_{i=0}^nf_i\pi^i+g\pi^{n+1}\] where $\pi$ is a uniformiser of $K$ (also a uniformiser of $L$), $f_0,\dots,f_n$ belong to the image of $H$, and $g\in\roi_L$. Applying the universal derivative obtains \[df=\sum_{i=0}^n\frac{df_i}{dt}\pi^idt+\pi^{n+1}dg,\] and computing $\frac{df}{dt}$ gives something similar. We again let $n\to\infty$ to deduce that $df=\frac{df}{dt}dt$ in $\Omega_{\roi_L/\roi_K}^\sub{sep}$. This completes the proof of our claim.

This proves that $dt$ generates $\Omega_{\roi_L/\roi_K}^\sub{sep}$, so we must now prove that it is not torsion. But the derivative $\frac{d}{dt}$ induces an $\roi_L$-linear map $\Omega_{\roi_L/\roi_K}\to \roi_L$ which descends to the maximal separated quotient and send $dt$ to $1$; this is enough. This completes the proof in the case $k_L=K$.

Now consider the general case $k_L\ge K$. Using the isomorphism $L\cong k_L\dblecurly{t}$, we set $M=K\dblecurly{t}$. The inclusions $\roi_K\le\roi_M\le\roi_L$, lemma \ref{lemma_differential_forms_of_tower}, and the first case of this proof applied to $K=k_M$, give an exact sequence of differential forms \begin{equation} 0\to\Omega_{\roi_M/\roi_K}^\sub{sep}\otimes_{\roi_M}\roi_L\to\Omega_{\roi_L/\roi_K}^\sub{sep}\to\Omega_{\roi_L/\roi_M}\to 0.\label{2}\end{equation} Furthermore, the isomorphism $L\cong M\otimes_Kk_L$ restricts to an isomorphism $\roi_L\cong\roi_M\otimes_{\roi_K}\roi_{k_L}$, and base change for differential forms gives $\Omega_{\roi_L/\roi_M}\cong\Omega_{\roi_{k_L}/\roi_K}\otimes_{\roi_{k_L}}\roi_L$; this isomorphism is given by the composition \[\Omega_{\roi_{k_L}/\roi_K}\otimes_{\roi_{k_L}}\roi_L\to\Omega_{\roi_L/\roi_K}\to\Omega_{\roi_L/\roi_M}.\] But this factors through $\Omega_{\roi_L/\roi_K}^\sub{sep}$, which splits (\ref{2}) and completes the proof.
\end{proof}

We may now define the {\em relative residue map for $L/K$} similarly to the equal characteristic case: \[\RES_L:\Omega_{L/K}^\sub{cts}\to k_L,\quad\omega=fdt\mapsto-\mbox{coeft}_{t^{-1}}(f)\] where the notation means that we expand $f$ in $k_L\dblecurly{t}$ and take the coefficient of $t^{-1}$. Implicit in the definition is the choice of an isomorphism $L\cong k_L\dblecurly{t}$ fixing $k_L$. The twist by $-1$ is necessary for the future reciprocity laws.

\begin{proposition}\label{proposition_residue_map_well_defined}
$\RES_L$ is well-defined, i.e. it does not depend on the chosen isomorphism $L\cong k_L\dblecurly{t}$.
\end{proposition}
\begin{proof}
As we noted in remark \ref{remark_local_parameter_determines_isomorphism}, the chosen isomorphism is determined uniquely by the choice of first local parameter. Let $T\in\roi_L$ be another first local parameter. Using a similar lifting argument (which simulates continuity) to that in the first half of the previous lemma, it is enough to prove \[\mbox{coeft}_{t^{-1}}\left(T^i\frac{dT}{dt}\right)=\begin{cases}1 & \mbox{if }i=-1,\\ 0 & \mbox{if }i\neq -1.\end{cases}\] Well, when $i\neq-1$, then $T^i\frac{dT}{dt}=\frac{d}{dt}(i^{-1}T^{i+1})$, which has $t^{-1}$ coefficient $0$, since this is true for the derivative of any element.

Now, the image of $T$ in $\res{L}$ has the form $\res{T}=\sum_{i=1}^{\infty}\theta_i\res{t}^i$, with $\theta_i\in\res{k}_L$ and $\theta_1\neq0$. Hence $T\equiv\sum_{i=1}^{\infty}a_it^i\mod\frak{p}_L$, where each $a_i\in k_L$ is a lift of $\theta_i$. Expanding the difference, a principal unit, as an infinite product obtains \[T=\left(\sum_{i=1}^{\infty}a_it^i\right)\prod_{j=1}^{\infty}(1+b_j\pi^j),\] for some $b_j\in\roi_L$, with $\pi$ a uniformiser of $k_L$ (also a uniformiser of $L$); we should remark that the above summation is a formal sum in $L\cong k_L\dblecurly{t}$, while the product is a genuinely convergent product in the valuation topology on $L$.

The map \[\mult{L}\to k_L,\quad \al\mapsto\mbox{coeft}_{t^{-1}}\left(\al^{-1}\frac{d\al}{dt}\right)\] is a continuous (with respect to the valuation topologies) homomorphism, so to complete the proof it is enough to verify the identities \[\mbox{coeft}_{t^{-1}}\left(\al^{-1}\frac{d\al}{dt}\right)=\begin{cases} 1&\mbox{if }\al=\sum_{i=1}^{\infty}a_it^i,\\ 0&\mbox{if }\al=1+b_j\pi^j.\end{cases}\] The first of these identities follows exactly as in the equal characteristic case of subsection \ref{subsection_equal_characteristic}. For the second identity, we compute as follows:
\begin{align*}
(1+b_j\pi^j)^{-1}\frac{d}{dt}(1+b_j\pi^j)
	&=(1-b_j\pi^j+b_j^2\pi^{2j}+\dots)\frac{db_j}{dt}\pi^j\\
	&=\frac{db_j}{dt}\pi^j-\frac{d(2^{-1}b_j^2)}{dt}\pi^{2j}+\frac{d(3^{-1}b_j^3)}{dt}\pi^{3j}+\dots
\end{align*}
This is a convergent sum, each term of which has no $t^{-1}$ coefficient; the proof is complete.
\end{proof}

We now establish the functoriality of residues with respect to the trace map:

\begin{proposition}
Suppose that $L'$ is a finite extension of $L$, and that $L'$ is also standard. Then the following diagram commutes:
\[\begin{CD}
\Omega_{L'/K}^\sub{cts} &@>\RES_{L'}>>&k_{L'}\\
@V\Tr_{L'/L}VV && @VV\Tr_{k_{L'}/k_L}V\\
\Omega_{L/K}^\sub{cts} &@>\RES_{L}>>&k_L
\end{CD}\]
\end{proposition}
\begin{proof}
Using the intermediate extension $Lk_{L'}$, we may reduce this to two cases: when we have compatible isomorphisms $L\cong k_L\dblecurly{t}$, $L\cong k_{L'}\dblecurly{t}$, or when $k_L=k_{L'}$. The first case is straightforward, so we only treat the second.

By the usual `principle of prolongation of algebraic identities' trick \cite[II.13]{Serre1988} we may reduce to the case $L\cong k_L\dblecurly{t}$, $L\cong k_L\dblecurly{T}$ with $t=T^e$. The same argument as in the equal characteristic case [loc. cit.] is then easily modified.
\end{proof}

\subsubsection{Extending the residue map to non-standard fields.}
Now suppose that $L$ is a two-dimensional local field of mixed characteristic which is not necessarily standard, and as usual fix a local field $K\le L$. Choose a standard subfield $M$ of $L$ with the same constant subfield as $L$ and of which $L$ is a finite extension; this is possible by lemma \ref{lemma_existence_of_standard_subfield}. Attempt to define the {\em relative residue map for $L/K$} to be composition \[\RES_L:\Omega_{L/K}^\sub{cts}\stackrel{\Tr_{L/M}}{\To}\Omega_{M/K}^\sub{cts}\stackrel{\RES_M}{\To}k_M=k_L.\]

\begin{lemma}
$\RES_L$ is independent of the choice of $M$.
\end{lemma}
\begin{proof}
Suppose that $M'$ is another field with the same properties as $M$, and let $\omega\in\Omega_{L/K}^\sub{cts}$. By an important structure result for two-dimensional local fields of mixed characteristic \cite[Theorem 2.1]{Zhukov1995} there is a finite extension $L'$ of $L$ such that $L'$ is standard. Using functoriality for standard fields, we have \[\RES_M(\Tr_{L/M}\omega)=|L':L|^{-1}\RES_M(\Tr_{L'/M}\omega)=|L':L|^{-1}\Tr_{k_{L'}/k_L}(\RES_{L'}(\omega))\] (here we have identified $\omega$ with its image in $\Omega_{L'/K}^\sub{cts}$). Since this expression is equally valid for $M'$ in place of $M$, we have proved the desired independence.
\end{proof}

The definition of the residue in the general case is chosen to ensure that functoriality still holds:

\begin{proposition}\label{proposition_functoriality_for_mixed_local_fields}
Let $L'/L$ be a finite extension of two-dimensional local fields of mixed characteristic; then the following diagram commutes
\[\begin{CD}
\Omega_{L'/K}^\sub{cts} @>\RES_{L'}>> k_{L'}\\
@V\Tr_{L'/L}VV  @VV\Tr_{k_{L'}/k_L}V\\
\Omega_{L/K}^\sub{cts} @>\RES_L>> k_L\\
\end{CD}\]
\end{proposition}
\begin{proof}
Let $M$ be a standard subfield of $L$ used to define $\RES_L$; then $M'=Mk_{L'}$ may be used to define $\RES_{L'}$. For $\omega\in\Omega_{L'/K}^\sub{cts}$, we have \[\RES_L(\Tr_{L'/L}\omega)=\RES_M(\Tr_{L/M}\Tr_{L'/L}\omega)=\RES_M(\Tr_{M'/M}\Tr_{L'/M'}\omega).\] Apply functoriality for standard fields to see that this equals \[\Tr_{k_{L'}/k_L}(\RES_{M'}(\Tr_{L'/M'}\omega))=\Tr_{k_{L'}/k_L}\RES_{L'}(\omega),\] as required.
\end{proof}

\subsubsection{Relation of the residue map to that of the residue fields.}
We finish this study of residues by proving that the residue map on a mixed characteristic, two-dimensional local field $L$ lifts the residue map of the residue field $\res{L}$. More precisely, we claim that the following diagram commutes
\[\begin{CD}
\Omega_{\roi_L/\roi_K}^\sub{sep}& @>\RES_L>> &\roi_{k_L}\\
@VVV&&@VVV\\
\Omega_{\res{L}/\res{K}}&@>e(L/k_L)\RES_{\res{L}}>>&\res{k}_L\\
\end{CD}\]
where some of the arrows deserve further explanation. The lower horizontal arrow is $e(L/k_L)$ times the usual residue map for $\res{L}$ (a local field of finite characteristic); note that $\res{K}$ is a finite subfield of $\res{L}$, and that $\res{k}_L$ is the constant subfield of $\res{L}$, which we identify with the residue field of $\res{L}$. Also, the top horizontal arrow is really the composition $\Omega_{\roi_L/\roi_K}^\sub{sep}\stackrel{j}\to\Omega_{L/K}^\sub{cts}\stackrel{\RES_L}{\to}k_L$; part of our claim is that $\RES_L\circ j$ has image in $\roi_{k_L}$.

Combining the identifications $\Omega_{\res{L}/\res{K}}=\Omega_{\res{L}/\res{k}_L}$ and $\Omega_{L/K}^\sub{cts}=\Omega_{L/k_L}^\sub{cts}$ with the natural surjection $\Omega_{\roi_L/\roi_K}^\sub{sep}\onto\Omega_{\roi_L/\roi_{k_L}}^\sub{sep}$, the problem is easily reduced to the case $K=k_L$, which we now consider.

Let us first suppose that $L$ is a standard field (so that $e(L/K)=1$); write $L=M$ for later clarity, and let $t\in M$ be a first local parameter. Then $\Omega_{\roi_M/\roi_K}^\sub{sep}=\roi_M\,dt$ by lemma \ref{lemma_structure_of_differential_forms_standard_fields} and so the image of $\Omega_{\roi_M/\roi_K}^\sub{sep}$ inside $\Omega_{M/K}^\sub{cts}=M\,dt$ is $\roi_M\,dt$. We need to show that $\RES_{\res{M}}(\res{f}\,d\res{t})=\res{\RES_M(f\,dt)}$ for all $f\in\roi_M$; this is clear from the explicit definition of the residue map for $M=K\dblecurly{t}$.

Now suppose $L$ is arbitrary, choose a first local parameter $t\in\roi_L$, and then choose a standard subfield $M$ such that $\res{M}=\res{L}$, $k_M=K$, and $t\in M$ (see lemma \ref{lemma_existence_of_standard_subfield}). To continue the proof, we must better understand the structure of $\Omega_{\roi_L/\roi_K}^\sub{sep}$. Let $\pi_L$ denote a uniformiser of $L$, so that $\roi_L=\roi_M[\pi_L]$; let $f(X)\in\roi_M[X]$ be the minimal polynomial of $\pi_L$, and write $f(X)=\sum_{i=0}^nb_iX^i$. We have our usual exact sequence \[0\to\Omega_{\roi_M/\roi_K}^\sub{sep}\otimes_{\roi_M}\roi_L\to\Omega_{\roi_L/\roi_K}^\sub{sep}\to\Omega_{\roi_L/\roi_M}\to0,\] so that $\Omega_{\roi_L/\roi_K}^\sub{sep}$ is generated by $dt$ and $d\pi_L$. Moreover, \begin{equation}0=d(f(\pi_L))=f'(\pi_L)\,d\pi_L+c\,dt,\label{3}\end{equation} where $c=\sum_{i=0}^n\frac{db_i}{dt}\pi_L^i$. Furthermore, using our exact sequence to see that $dt$ is not torsion, and from the fact that $\Omega_{\roi_L/\roi_M}\cong\roi_L/\langle f'(\pi_L)\rangle$ (using the generator $d\pi_L$), it is easy to check that (\ref{3}) is the only relation between the generators $dt$ and $d\pi_L$.

We now define a trace map $\Tr_{\roi_L/\roi_M}:\Omega_{\roi_L/\roi_K}^\sub{sep}\to\Omega_{\roi_M/\roi_K}^\sub{sep}$ as follows: \begin{align*}\Tr_{\roi_L/\roi_M}(a\,d\pi_L)&=\Tr_{L/M}(-acf'(\pi_L)^{-1})\,dt\\\Tr_{\roi_L/\roi_M}(b\,dt)&=\Tr_{L/M}(b)\,dt\end{align*} for $a,b\in\roi_L$. It is important to recall the classical different formula (\cite[III.2]{Neukirch1999}; also see appendix \ref{section_finite_morphisms_etc}): \[f'(\pi)^{-1}\roi_L=\frak{C}(\roi_L/\roi_M)\;\;(=\!\{x\in L:\Tr_{L/M}(x\roi_L)\subseteq\roi_F\}),\] to see that this is well-defined. Furthermore, if we base change $-\otimes_{\roi_L}L$, then we obtain the usual trace map $\Tr_{L/M}:\Omega_{L/K}^\sub{cts}\to\Omega_{M/K}^\sub{cts}$.

By definition of the residue map on $L$, it is now enough to show that the diagram
\[\begin{CD}
\Omega_{\roi_L/\roi_K}^\sub{sep}&@>\Tr_{\roi_L/\roi_M}>>&\Omega_{\roi_M/\roi_K}^\sub{sep}\\
@VVV&&@VVV\\
\Omega_{\res{L}/\res{K}}&@>\times|L:M|>>&\Omega_{\res{M}/\res{K}}
\end{CD}\]
commutes. Well, for an element of the form $b\,dt\in\Omega_{\roi_L/\roi_K}^\sub{sep}$ with $b\in\roi_L$, commutativity is clear. Now consider an element $a\,d\pi_L\in\Omega_{\roi_L/\roi_K}^\sub{sep}$; the image of this in $\Omega_{\res{L}/\res{K}}$ is zero, so we must show that $\res{\Tr_{\roi_L/\roi_M}(a\,d\pi_L)}=0$ in $\Omega_{\res{M}/\res{K}}$. For this we recall another formula relating the trace map and different, namely \[\Tr_{L/M}(\pi_L^if'(\pi_L)^{-1}\roi_L)=\pi_M^{\lfloor\frac{i}{e}\rfloor}\roi_M,\] where $i\in\bb{Z}$, $e=|L:M|$, and $\lfloor\phantom{a}\rfloor$ denotes the greatest integer below (see e.g. \cite[Proposition III.1.4]{Fesenko2002}). Since $f$ is an Eisenstein polynomial, $\nu_L(\frac{da_i}{dt})\ge e$ for all $i$, and so $\nu_L(c)\ge e$; by the aforementioned formula, $\Tr_{L/M}(cf'(\pi_L)^{-1}\roi_L)\subseteq\pi_M\roi_M$. This is what we needed to show, and completes the proof of compatibility between $\RES_L$ and $\RES_{\res{L}}$.

\begin{corollary}
Let $L$ be a two-dimensional local field of mixed characteristic, and $K\le L$ a local field. Then the following diagram commutes:
\[\begin{CD}
\Omega_{\roi_L/\roi_K}^\sub{sep}& @>\Tr_{k_L/K}\circ\RES_L>> &\roi_{K}\\
@VVV&&@VVV\\
\Omega_{\res{L}/\res{K}}&@>e(L/K)\Tr_{\res{k}_L/\res{K}}\circ\RES_{\res{L}}>>&\res{K}\\
\end{CD}\]
\end{corollary}
\begin{proof}
It is enough to combine what we have just proved with the commutativity of
\[\begin{CD}
\roi_{k_L}&@>\Tr_{k_L/K}>>&\roi_K\\
@VVV&&@VVV\\
\res{k}_L&@>e(k_L/K)\Tr_{\res{k}_L/\res{K}}>>&\res{K}
\end{CD}\]
\end{proof}

\section{Reciprocity for two-dimensional, normal, local rings}\label{section_reciprocity_for_rings}
Now we consider a semi-local situation and prove the promised reciprocity law (theorem \ref{theorem_incomplete_reciprocity}).

Let $A$ be a two-dimensional, normal, complete, local ring of characteristic zero, with finite residue field of characteristic $p$; for the remainder of this section, we will refer to these collective conditions as (\dag). Denote by $F$ the field of fractions of $A$ and by $\frak{m}_A$ the maximal ideal. For each height one prime $y\lhd A$ (we will sometimes write $y\lhd\!^1A$), the localisation $A_y$ is a discrete valuation ring, and we denote by $F_y=\Frac\hat{A_y}$ the corresponding complete discrete valuation field. The residue field of $F_{y}$ is $\res{F}_{y}=\Frac A/y$. Moreover, $A/y$ is a one-dimensional, complete, local domain, and so its field of fractions is a complete discrete valuation field whose residue field is a finite extension of the residue field of $A/y$, which is the same as the residue field of $A$. Therefore $F_y$ is a two-dimensional local field of characteristic zero.

Since $A$ is already complete, there is no confusion caused by writing $\hat{A}_y$ instead of $\hat{A_y}$ (notice the different sized hats).

\begin{lemma}\label{lemma_unique_embedding_of_Zp_into_local_ring}
There is a unique ring homomorphism $\bb{Z}_p\to A$, and it is a closed embedding.
\end{lemma}
\begin{proof}
The natural embedding $j:\bb{Z}\into A$ is continuous with respect to the $p$-adic topology on $\bb{Z}$ and the $\frak{m}_A$-adic topology on $A$ since $p^n\bb{Z}_p\subseteq j^{-1}(\frak{m}_A^n)$ for all $n\ge 0$. Therefore $j$ extends to a continuous injection $j:\bb{Z}_p\into A$, which is a closed embedding since $\bb{Z}_p$ is compact and $A$ is Hausdorff.

Now suppose that $\phi:\bb{Z}_p\to A$ is an arbitrary ring homomorphism. Then $\phi^{-1}(\frak{m}_A^n)$ is an ideal of $\bb{Z}_p$ which contains $p^n\bb{Z}$; but every ideal of $\bb{Z}_p$ is closed, and so it contains $p^n\bb{Z}_p$. Therefore $\phi$ is continuous; since $\phi$ agrees with $j$ on $\bb{Z}$, they are equal.
\end{proof}

We fix a finite extension $\roi_K$ of $\bb{Z}_p$ inside $A$, where $\roi_K$ is the ring of integers of a finite extension $K$ of $\bb{Q}_p$. For each height one prime $y\lhd A$, we have $K\le F_y$, and the constant/coefficient field $k_y=k_{F_y}$ of $F_y$ is a finite extension of $K$. There is a natural map $\Omega_{F/K}\to\Omega_{F_y/K}^\sub{cts}$, so we may define the {\em residue map at $y$} by \[\RES_y:\Omega_{F/K}\To \Omega_{F_y/K}^\sub{cts}\stackrel{\RES_{F_y}}{\To}k_y.\] It is a nuisance having the residue maps associated to different primes taking values in different finite extensions of $K$, so we also introduce \[\RRES_y=\Tr_{k_y/K}\RES_y:\Omega_{F/K}\to K.\]

Our immediate aim, to be deduced in several stages, is the following reciprocity law:

\begin{theorem}
Let $\omega\in\Omega_{F/K}$; then for all but finitely many height one primes $y\lhd A$ the residue $\RES_y(\omega)$ is zero, and \[\sum_{y\lhd^1 A}\RRES_y(\omega)=0\] in $K$.
\end{theorem} 

We will also prove an analogous result without the assumption that $A$ is complete; see theorem \ref{theorem_incomplete_reciprocity}.

\subsection{Reciprocity for $\roi_K[[T]]$}
We begin by establishing reciprocity for $B=\roi_K[[T]]$. More precisely, we shall consider $B\cong\roi_K[[T]]$; although this may seem to be a insignificant difference, it is important to understand the intrinsic role of $T$, especially for the proof of proposition \ref{proposition_residue_map_well_defined}.

\begin{lemma}
Let $B$ satisfy conditions (\dag) and also be regular; let $\roi_K\le B$ be the ring of integers of a local field, and assume that $\res{K}=k(\frak{m}_B)$ and $\frak{p}_K\not\subseteq\frak{m}_B^2$. Let $\pi_K$ be any prime of $K$.

Then there exists $t\in\frak{m}_B$ such that $\frak{m}_B=\langle \pi_K,t\rangle$. If $t$ is any such element, then each $f\in B$ may be uniquely written as a convergent series $f=\sum_{i=0}^{\infty}a_i t^i$, with $a_i\in\roi_K$, and this defines an $\roi_K$-isomorphism $B\cong\roi_K[[T]]$, with $t\mapsto T$.
\end{lemma}
\begin{proof}
Since $\pi_K$ is non-zero in the $k(\frak{m}_B)$-vector space $\frak{m}_B/\frak{m}_B^2$, which has dimension two by regularity, there is $t\in B$ such that (the images of) $\pi_K,t$ are a basis for this space; hence $\frak{m}_A=\langle\pi_K,t\rangle$ by Nakayama's lemma.

Now, $B/tB$ is a one-dimensional, complete, regular, local ring, i.e. a complete discrete valuation ring, in which $\pi_K$ is prime. Since $tB$ is prime, $tB\cap\roi_K=\{0\}$ and so $\roi_K\into B/tB$; but these two complete discrete valuation fields have the same prime and same residue field, hence are equal.

Any series of the given form converges in $B$ because $B$ is complete and $a_it^i$ belongs to $\frak{m}_B^i$. Conversely, for any $f\in B$ we may write $f\equiv a_0\mod{tB}$ for some $a_0\in\roi_K$ (since $B/tB=\roi_K$); then replace $f$ by $t^{-1}(f-a_0)$ and repeat the process to obtain the desired expansion for $f$. If a series $\sum_{i\ge I}a_it^i$ is zero, with $a_I\neq 0$, then we get $a_It^I\in t^{I+1}B$, which contradicts the identity $tB\cap\roi_K=\{0\}$.
\end{proof}

Now let $B$, $\roi_K$, $\pi_K$, $t$ satisfy the conditions of the previous lemma; set $M=\Frac{B}$. Using the isomorphism $B\cong\roi_K[[T]]$, we may describe the height one primes $y$ of $B$ (see e.g. \cite[Lemma 5.3.7]{Neukirch2000}):
\begin{enumerate}
\item $p\in y$. Then $y=\pi_KB$, and $M_y$ is a two-dimensional local field of mixed characteristic which is $K$-isomorphic to $K\dblecurly{t}$ and has constant field $k_y=K$.
\item $p\notin y$. Then $y=hB$, where $h\in\roi_K[t]$ is an irreducible, Weierstrass polynomials (i.e. $h=t^d+a_{d-1}t^{d-1}+\dots+a_0$, with $a_i\in\frak{p}_K$), and $M_y$ is a two-dimensional local field of equal characteristic. The coefficient field $k_y$ is the finite extension of $K$ generated by a root of $h$. Finally, $M_y$ is $k_y$-isomorphic to $k_y((t_y))$, where $t_y$ is a uniformiser at $y$, e.g.~$t_y=h$.
\end{enumerate}

We need a convenient set of additive generators of $M$:

\begin{lemma}\label{lemma_additive_generators_of_field}
Each element of $M$ is a finite sum of elements of the form \[\frac{\pi_K^ng}{h^r},\] with $h\in\roi_K[t]$ an irreducible, Weierstass polynomial, $r>0$, $n\in\bb{Z}$, and $g\in B$.
\end{lemma}
\begin{proof}
We begin with an element of $M$ of the form $1/(\pi_1^{r_1}\pi_2^{r_2})$, with $\pi_1,\pi_2$ distinct irreducible elements of $B$, and $r_1,r_2\ge 1$. Set $I=\langle\pi_1^{r_1},\pi_2^{r_2}\rangle\lhd B$; then $I$ cannot be contained inside any height one prime $\frak{q}$, for then $\frak{q}$ would contain both $\pi_1$ and $\pi_2$. So $\sqrt{I}=\frak{m}_B$ and therefore $\frak{m}_B^m\subseteq I$ for $m\gg0$. Therefore we may write $\pi_K^m=g_1\pi_1^{r_1}+g_2\pi_2^{r_2}$ for some $g_1,g_2\in B$, and we deduce \[\frac{1}{\pi_1^{r_1}\pi_2^{r_2}}=\frac{\pi_K^mg_1}{\pi_2^{r_2}}+\frac{\pi_K^mg_2}{\pi_1^{r_1}}.\]

Now, a typical element of $M$ has the form $a/b$, with $a,b\in B$. Since $B$ is a unique factorisation domain whose prime ideals are as described above, we may write $b=u\pi_K^rh_1^{r_1}\cdots h_s^{r_s}$ where $u\in\mult{B}$, the $h_i$ are irreducible Weierstass polynomials, and all the exponents are $\ge0$. Replacing $a$ with $u^{-1}a$, we may suppose $u=1$. Applying the first part of the proof repeatedly decomposes $a/b$ into a sum of the required form.
\end{proof}

We also need to understand the space of relative differential forms:

\begin{lemma}\label{lemma_2d_complete_ring_base_case}
$\Omega_{B/\roi_K}^\sub{sep}$ is a free $B$-module of rank one, with basis $dt$. For each height one prime $y\lhd B$, the natural map $\Omega_{B/\roi_K}\otimes_B \hat{B}_y\to\Omega_{\hat{B}_y/\roi_K}$ descends to an isomorphism \[\Omega_{B/\roi_K}^\sub{sep}\otimes_B \hat{B}_y\isoto\Omega_{\hat{B}_y/\roi_K}^\sub{sep}.\] Hence there is an induced isomorphism $\Omega_{B/\roi_K}^\sub{sep}\otimes_BM_y\isoto\Omega_{M_y/K}^\sub{cts}$.
\end{lemma}
\begin{proof}
The first claim may be proved in an identical way to lemma \ref{lemma_differential_forms_of_K[[t]]}. Alternatively, use remark \ref{remark_completion_of_fin_gen_algebra} to deduce that $\Omega_{B/\roi_K}^\sub{sep}=\Omega_{\roi_K[t]/\roi_K}\otimes_{\roi_K[t]}B$.

If $y$ is a height one prime of $B$ then there is a natural map
\begin{align*}\phi:\Omega_{B/\roi_K}^\sub{sep}\otimes_B\hat{B}_y
	&=(\Omega_{\roi_K[t]/\roi_K}\otimes_{\roi_K[t]}B)\otimes_B\hat{B}_y\\
	&=\Omega_{\roi_K[t]/\roi_K}\otimes_{\roi_K[t]}\hat{B}_y\to\Omega_{\hat{B}_y/\roi_K}\onto\Omega_{\hat{B}_y/\roi_K}^\sub{sep},\end{align*}
and we shall now construct the inverse of $\phi$. Define an $\roi_K$-derivation of $B_y$ by \[d_1:B_y\to\Omega_{B/\roi_K}^\sub{sep}\otimes_B\hat{B}_y\,\quad b/s\mapsto db\otimes s^{-1}-b\,ds\otimes s^{-2}\] where $b\in B$, $s\in B\setminus y$ (this is well-defined). Moreover, the right hand side is a finite $\hat{B}_y$-module, hence is complete and separated for the $y$-adic topology; so $d_1$ (which is easily seen to be $y$-adically continuous) extends from $B_y$ to $\hat{B}_y$. This derivation then induces a homomorphism of $\hat{B}_y$-modules $\Omega_{\hat{B}_y/\roi_K}\to\Omega_{B/\roi_K}^\sub{sep}\otimes_B\hat{B}_y$, and this descends to \[\psi:\Omega_{\hat{B}_y/\roi_K}^\sub{sep}\to\Omega_{B/\roi_K}^\sub{sep}\otimes_B\hat{B}_y\] since $\Omega_{B/\roi_K}^\sub{sep}\otimes_B\hat{B}_y$ is a finitely generated $\hat{B}_y$-module.

It is immediate that $\psi\phi=\mbox{id}$. It is also easy to see that $\phi\psi(db)=db$ for any $b\in B_y$; since such elements are dense in the Hausdorff space $\Omega_{\hat{B}_y/\roi_K}^\sub{sep}$, we deduce $\phi\psi=\mbox{id}$.
\end{proof}

In particular, we now know that the residue map at $y$, initially defined on $\Omega_{M/K}$, factors through its quotient $\Omega_{B/\roi_K}^\sub{sep}\otimes_B M$. We may now prove reciprocity for $B$:

\begin{theorem}\label{theorem_reciprocity_base_case}
For each $\omega\in\Omega_{B/\roi_K}^\sub{sep}\otimes_B M$, the local residue $\RES_y\omega$ is zero for all but finitely many $y\lhd\!^1B$, and \[\sum_{y\lhd\!^1B}\RRES_y\omega=0\] in $K$.
\end{theorem}
\begin{proof}
By lemmas \ref{lemma_additive_generators_of_field} and \ref{lemma_2d_complete_ring_base_case}, it is enough to consider the case $\omega=f\,dt$ with \[f=\frac{\pi_K^ng}{h^r},\] where $h,r,n,g$ are as in lemma \ref{lemma_additive_generators_of_field}.

Let $y=t_yB$ be a prime with $t_y$ an irreducible, Weierstrass polynomial. If $t_y\neq h$, then $\pi_K^ng/h^r$ and $t$ both belong to $B_y$, and so \[\mbox{coeft}_{t_y^{-1}}\left(\frac{\pi_K^ng}{h^r}\frac{dt}{dt_y}\right)=0\] by a basic property of the residue map; that is, $\RES_y(\omega)=0$. This establishes our first assertion. For the remainder of the proof, set $y=hA$; we must prove that \begin{equation}\RRES_y(\omega)+\RRES_{\pi_KB}(\omega)=0. \label{4}\end{equation}

Suppose for a moment that $g$ belongs to $\roi_K[t]$, and consider the rational function field $K(t)\le M$. For any point $x$ of $\bb{P}_K^1$, let $K(t)_x$ be the completion of $K(t)$ at the place $x$; then $K(t)_x$ is a two-dimensional local field of equal characteristic. Let $k_x$ denote its unique coefficient field containing $K$, and let $\RES_x:\Omega_{L_x/K}^\sub{cts}\to k_x$ denote the residue map. By the assumption on $g$ we have  $\omega\in\Omega_{K(t)/K}$, and global reciprocity for $\bb{P}_K^1$ implies that $\sum_{x\in\bb{P}_K^1}\Tr_{k_x/K}\RES_x(\omega)=0$.

Furthermore, an argument as at the start of this proof proves that $\RES_x(\omega)=0$ unless $x$ corresponds to the irreducible polynomial $h$, or $x=\infty$. Moreover, in the first case, $K(t)_x=M_y$, $k_x=k_y$, and $\RES_x(\omega)=\RES_y(\omega)$. Therefore to complete the proof (with $g$ still a polynomial) it is necessary and sufficient to show that \begin{equation}\RES_{\pi_KB}(\omega)=\RES_{\infty}(\omega).\label{5}\end{equation} Note that the residue map on the left is for a two-dimensional local field of mixed characteristic, while that on the right is for one of equal characteristic. This passage between different characteristics is the key to the proof.

To prove (\ref{5}), write $t_{\infty}=t^{-1}$, which is a local parameter at $\infty$, and expand $h^{-r}$ in $K(t)_{\infty}=K((t_{\infty}))$ as $h^{-r}=\sum_{i\ge I} a_it_{\infty}^i$, say. Since $h^r$ is a Weierstrass polynomial, it is easily checked that $a_i\to 0$ in $K$ as $i\to\infty$; therefore the series $\sum_{i\le -I} a_{-i}t^i$ is a well-defined element of $M_{\pi_kB}=K\dblecurly{t}$. Moreover, since multiplication in both $K\dblecurly{t}$ and $K((t_{\infty}))$ is given by formal multiplication of series, we deduce \[h^r\sum_{i\le -I} a_{-i}t^i=1,\] i.e. $\sum_{i\le -I} a_{-i}t^i$ is the series expansion of $h^{-r}$ in $M_{\pi_kB}=K\dblecurly{t}$. Now let $\sum_ib_it_{\infty}^i$ be the expansion of $\pi_K^ng/h^r$ of $K(t)_{\infty}$; then $\sum_i b_{-i}t^i$ is the formal expansion of $\pi_K^ng/h^r$ in $M_{\pi_kB}$, and so
\begin{align*}
\RES_{\infty}\left(\frac{\pi_K^ng}{h^r}dt\right)
	&=\mbox{coeft}_{t_{\infty}^{-1}}\left(\frac{\pi_K^ng}{h^r}\frac{dt}{dt_{\infty}}\right)\\
	&=\mbox{coeft}_{t_{\infty}^{-1}}\left(-t_{\infty}^{-2}\sum_ib_it_{\infty}^i\right)\\
	&=-b_1\\
	&=-\mbox{coeft}_{t^{-1}}\sum_i b_{-i}t^i\\
	&=\RES_{\pi_KB}\left(\frac{\pi_K^ng}{h^r}dt\right).
\end{align*}

This completes the proof of identity (\ref{4}) for $g\in\roi_K[t]$; to prove it in general and complete the proof, it is enough to check that both sides of (\ref{4}) are continuous functions of $g$, with respect to the $\frak{m}_B$-adic topology on $B$ and the discrete valuation topology on $K$. This is straightforward, though tedious, and so we omit it.
\end{proof}

\subsection{Reciprocity for complete rings}
Now we extend the reciprocity law to the general case. Fix both a ring $A$ satisfying conditions (\dag) (see the start of section \ref{section_reciprocity_for_rings}) and the ring of integers of a local field $\roi_K\le A$. Reciprocity for $A$ will follows in the usual way by realising $A$ as a finite extension of $\roi_K[[T]]$:

\begin{lemma}
There is a ring $B$ between $\roi_K$ and $A$ which is $\roi_K$-isomorphic to $\roi_K[[T]]$, and such that $A$ is a finite $B$-module.
\end{lemma}
\begin{proof}
By \cite[Theorem 16]{Cohen1946}, $A$ contains a subring $B_0$, over which it is a finitely generated module, and such that $B_0$ is a two-dimensional, $p$-adic ring with residue field equal to that of $A$. Supposing that this residue field is $\bb{F}_q$, we therefore have an isomorphism $i:\bb{Z}_q[[T]]\isoto B_0$. By the uniqueness of the embedding $\bb{Z}_p\into A$ (lemma \ref{lemma_unique_embedding_of_Zp_into_local_ring}), it follows that $i(\bb{Z}_p)\subseteq\roi_K$. Define \[j:\roi_K[[T]]=\bb{Z}_p[[T]]\otimes_{\bb{Z}_p}\roi_K\to A,\quad f\otimes\al\mapsto i(f)\al.\]

The kernel of $j$ is a prime ideal of $\roi_K[[T]]$ whose contraction to $\roi_K$ is zero. If the kernel is non-zero then there is an Eisenstein polynomial $h\in\roi_K[T]$ such that that $h(i(T))=0$ (this follows from the classification of prime ideals in $\roi_K[[T]]$ discussed earlier), suggesting that $i(T)$ is algebraic over $\roi_K$ and hence over $\bb{Z}_p$; this contradicts the injectivity of $i$. Hence $j$ is an isomorphism onto its image, as desired.
\end{proof}

Let $B$ be as given by the previous lemma, and write $M=\Frac B$, $F=\Frac A$. We now generalise lemma \ref{lemma_2d_complete_ring_base_case}. However, note that if $A$ can be written as the completion of a localisation of a finitely generated $\roi_K$-algebra, then the following proof can be significantly simplified, simply by imitating the proof of lemma \ref{lemma_2d_complete_ring_base_case}; see also lemma \ref{lemma_differential_forms_global_to_local}.

\begin{lemma}
$\Omega_{A/\roi_K}^\sub{sep}$ is a finitely generated $A$-module of rank $1$. For each height one prime $y\lhd A$, the natural map $\Omega_{A/\roi_K}\otimes_A \hat{A}_y\to\Omega_{\hat{A}_y/\roi_K}$ descends to an isomorphism \[\Omega_{A/\roi_K}^\sub{sep}\otimes_A \hat{A}_y\isoto\Omega_{\hat{A}_y/\roi_K}^\sub{sep}.\] Hence there is an induced isomorphism $\Omega_{A/\roi_K}^\sub{sep}\otimes_A F_y\stackrel{\simeq}{\to}\Omega_{F_y/K}^\sub{cts}$.
\end{lemma}
\begin{proof}
Lemmas \ref{lemma_2d_complete_ring_base_case} and \ref{lemma_differential_forms_of_tower} imply that there is a natural exact sequence \begin{equation}0\to\Omega_{B/\roi_K}^\sub{sep}\otimes_B A\to\Omega_{A/\roi_K}^\sub{sep}\to\Omega_{A/B}\to 0,\label{6}\end{equation} which proves the first claim since $\Omega_{A/B}$ is a finitely generated, torsion $A$-module. 

Now we are going to construct a commutative diagram with exact rows:
\[\minCDarrowwidth23pt\begin{CD}
0&@>>>&\Omega_{B/\roi_K}^\sub{sep}\otimes_B \hat{A}_y&@>>>&\Omega_{A/\roi_K}^\sub{sep}\otimes_A\hat{A}_y&@>>>&\Omega_{A/B}\otimes_A\hat{A}_y&@>>>& 0\\
&&&&@A\psi_B' A\cong A&&@A{\psi_A}AA&&@AA\cong A&\\
0&@>>>&\Omega_{\hat{B}_{y'}/\roi_K}^\sub{sep}\otimes_{\hat{B}_{y'}}\hat{A}_y&@>>>&\Omega_{\hat{A}_y/\roi_K}^\sub{sep}&@>>>&\Omega_{\hat{A}_y/\hat{B}_y}&@>>>& 0
\end{CD}\]
The top line is obtained by tensoring (\ref{6}) with $\hat{A}_y$. For the bottom row, set $y'=y\cap B$, use lemma \ref{lemma_2d_complete_ring_base_case} to see that $\Omega_{B_{y'}/\roi_K}^\sub{sep}$ is free and that we may therefore apply lemma \ref{lemma_differential_forms_of_tower} to the tower of rings $\hat{A}_y\ge\hat{B}_{y'}\ge\roi_K$. In lemma \ref{lemma_2d_complete_ring_base_case} we also constructed a natural map \[\psi_B=\psi:\Omega_{\hat{B}_y/\roi_K}^\sub{sep}\to\Omega_{B/\roi_K}^\sub{sep}\otimes_B\hat{B}_y;\] its definition did not use any special properties of $B$ and so we may similarly define $\psi_A$. Base change $\psi_B$ by $\hat{A}_y$ to obtain the isomorphism $\psi_B'$ in the diagram. Finally, one may see in a number of different ways that there is an isomorphism $\Omega_{A/B}\otimes_A\hat{A}_y\cong\Omega_{\hat{A}_y/\hat{B}_y}$ which is natural enough so that the diagram will commute.

It follows that $\psi_A$ is an isomorphism, as required.
\end{proof}

The previous lemma implies that $\Omega_{B/\roi_K}^\sub{sep}\otimes_BF\cong\Omega_{A/\roi_K}^\sub{sep}\otimes_AF$, and so we have natural trace maps \begin{align*}\Tr_{F/M}&:\Omega_{A/\roi_K}^\sub{sep}\otimes_AF\to\Omega_{B/\roi_K}^\sub{sep}\otimes_BM\\ \Tr_{F_Y/M_y}&:\Omega_{A/\roi_K}^\sub{sep}\otimes_AF_Y\to\Omega_{B/\roi_K}^\sub{sep}\otimes_B M_y,\end{align*} where $Y$ is a height one prime of $A$ and $y=Y\cap B$. Using these we establish the expected functoriality for our residue maps:

\begin{proposition}\label{proposition_functoriality_in_complete_case}
Let $y$ be a fixed height one prime of $B$. Then for all $\omega\in\Omega_{A/\roi_K}^\sub{sep}\otimes_A F$, we have \[\RRES_y\Tr_{F/M}(\omega)=\sum_{Y|y}\RRES_Y(\omega),\] where $Y$ ranges over the (finitely many) height one primes of $A$ sitting over $y$.
\end{proposition}
\begin{proof}
Set $A_y=A\otimes_BB_y=(B\setminus y)^{-1}A\subseteq F$. Then $A_y/B_y$ is a finite extension of Dedekind domains, with the maximal ideals of $A_y$ corresponding to the primes $Y$ of $A$ (necessarily of height one) sitting over $y$. Therefore, for any $x\in F$, one has the usual local-global trace formula $\Tr_{F/M}(x)=\sum_{Y|y}\Tr_{F_Y/M_y}(x)$. In terms of differential forms, \[\Tr_{F/M}\omega=\sum_{Y|y}\Tr_{F_Y/M_y}\omega\] for all $\omega\in\Omega_{A/\roi_K}^\sub{sep}\otimes_A F$. Applying $\RES_y$ to each side of this expression and using propositions \ref{proposition_functoriality_in_equal_characteristic_case} and \ref{proposition_functoriality_for_mixed_local_fields} obtains \[\RES_y\Tr_{F/M}(\omega)=\sum_{Y|y}\Tr_{k_Y/k_y}\RES_Y(\omega).\] Apply $\Tr_{k_y/K}$ to complete the proof.
\end{proof}

Our desired reciprocity for $A$ follows in the usual way:

\begin{theorem}\label{theorem_general_complete_reciprocity}
For each $\omega\in\Omega_{A/\roi_K}^\sub{sep}\otimes_A F$, the local residue $\RES_y\omega$ is zero for all but finitely many $y\lhd\!^1A$, and \[\sum_{y\lhd\!^1A}\RRES_y\omega=0.\]
\end{theorem}
\begin{proof}
Standard divisor theory implies that any $f\in\mult{F}$ belongs to $A_y$ for all but finitely many $y\lhd\!^1A$. If $fdg$ is a nonzero element of $\Omega_{A/\roi_K}^\sub{sep}\otimes_A F$, then $\RES_Yfdg=0$ for any $Y\lhd\!^1A$ which satisfies the following conditions: $p\notin Y$ and $f,g\in A_Y$. Since all but finitely many $Y$ satisfy these conditions, we have proved the first claim.

We may now complete the proof with the usual calculation, by reducing reciprocity via the previous proposition to the already-proved reciprocity for $B$:
\begin{align*}
\sum_{Y\lhd\!^1 A}\RRES_Y\omega
	&=\sum_{y\lhd\!^1 B}\sum_{Y|y}\RRES_Y\omega\\
	&=\sum_{y\lhd\!^1 B}\RRES_y(\Tr_{F/M}\omega)\\
	&=0.
\end{align*}
\end{proof}

\subsection{Reciprocity for incomplete rings}\label{subsection_reciprocity_for_incomplete_rings}
We have thus far restricted our attention to complete local rings; we will now remove the completeness hypothesis. We do not prove reciprocity in the fullest generality, but restrict to those rings which will later arise from an arithmetic surface. Let $\roi_K$ be a discrete valuation ring of characteristic zero and with finite residue field, and $A\ge\roi_K$ a two-dimensional, normal, local ring with finite residue field of characteristic $p$; assume further that $A$ is the localisation of a finitely generated $\roi_K$-algebra.

Since $A$ is excellent, its completion $\hat{A}$ is also normal; therefore $\hat{A}$ satisfies conditions (\dag) from the start of the section, and $\hat{\roi}_K\le\hat{A}$ is the ring of integers of a local field, as has appeared in the previous subsections. Write $F=\Frac\hat{A}$ and $\hat{K}=\Frac{\hat{\roi}_K}$.

The following global-to-local isomorphism is extremely useful for explicit calculations. Since the notation can look confusing, let us mention that if $Y$ is a height one prime of $\hat{A}$, then the completion of the discrete valuation ring $(\hat{A})_Y$ is denoted $\hat{\hat{A}_Y}$.

\begin{lemma}\label{lemma_differential_forms_global_to_local}
Let $Y$ be a height one prime of $\hat{A}$; then the natural map \[\Omega_{A/\roi_K}\otimes_A\hat{\hat{A}_Y}\to\Omega_{\hat{\hat{A}_Y}/\hat{\roi}_K}^\sub{sep}\] is an isomorphism.
\end{lemma}
\begin{proof}
One follows the proof of lemma \ref{lemma_2d_complete_ring_base_case} almost exactly, replacing $B$ by $\hat{A}$ and $\roi_K[t]$ by $A$. The only additional observation which needs to be made is that the universal derivation $d:\hat{A}\to\Omega_{\hat{A}/\roi_K}^\sub{sep}$ must be trivial on $\hat{\roi}_K$, and therefore $\Omega_{\hat{A}/\roi_K}^\sub{sep}=\Omega_{\hat{A}/\hat{\roi}_K}^\sub{sep}$.
\end{proof}

For $Y\lhd\!^1\hat{A}$, the previous lemma gives us a natural isomorphism \[\Omega_{A/\roi_K}\otimes_{A} F_Y\isoto\Omega_{F_Y/\hat{K}}^\sub{cts},\] and we thus pull back the relative residue map of $F_Y/\hat{K}$ to get \[\RES_Y:\Omega_{\Frac{A}/K}=\Omega_{A/\roi_K}\otimes_{A}\Frac{A}\to k_Y,\] where, as usual, $k_Y$ denote the coefficient/constant field of $F_Y$.

More importantly, if $y$ is instead a height one prime of $A$, then set \[\RRES_y=\sum_{Y|y}\Tr_{k_Y/\hat{K}}\RES_Y:\Omega_{\Frac{A}/K}\to\hat{K}\] where $Y$ ranges over the finitely many height one primes of $\hat{A}$ sitting over $y$.

We need a small lemma. We shall say that a prime of $\hat{A}$ is {\em transcendental} if and only if its contraction to $A$ is zero; such a prime has height one and does not contain $p$. The transcendental primes are artificial in a sense; they have pathological properties (e.g. if $Y$ is transcendental then $\Frac{A}\le\hat{A}_Y$) and do not contain interesting information about $A$.

\begin{lemma}
Let $Y$ be a height one prime of $\hat{A}$. If $Y$ is not transcendental then it is a prime minimal over $y\hat{A}$, where $y=A\cap Y$. On the other hand, if $Y$ is transcendental, then $\RES_Y(\omega)=0$ for all $\omega\in\Omega_{\Frac{A}/K}$.
\end{lemma}
\begin{proof}
Since $y=A\cap Y$ is non-zero by assumption, so is $y\hat{A}$. Since $y\hat{A}$ is contained in $Y$ there is a prime $P\lhd\hat{A}$ which is minimal over $y\hat{A}$ and which is contained in $Y$. But then $P\neq 0$ and we have a chain of primes $0\lhd P\unlhd Y$; since $Y$ has height $1$ in ${\hat{A}}$, we deduce $Y=P$, i.e. $Y$ is minimal over $y\hat{A}$.

If $fdg$ is a element of $\Omega_{\Frac{A}/K}$, then as we remarked above, $f$ and $g$ belong to $\hat{A}_Y$; therefore $\RES_Y(fdg)=0$, just as in the proof of theorem \ref{theorem_general_complete_reciprocity}.
\end{proof}

The reciprocity law for $A$ follows:

\begin{theorem}\label{theorem_incomplete_reciprocity}
For any $\omega\in\Omega_{\Frac A/K}$, the residue $\RRES_y(\omega)$ is non-zero for only finitely many $y\lhd\!^1 A$, and \[\sum_{y\lhd\!^1A}\RRES_y(\omega)=0\] in $\hat{K}$.
\end{theorem}
\begin{proof}
Immediate from theorem \ref{theorem_general_complete_reciprocity} and the previous lemma.
\end{proof}

\section{A reciprocity law for arithmetic surfaces}\label{section_reciprocity_for_surfaces}
Let $\roi_K$ be a Dedekind domain of characteristic zero and with finite residue fields; denote by $K$ its field of fractions. Let $X$ be a curve over $\roi_K$; more precisely, $X$ is a normal scheme, flat and projective over $S=\Spec\roi_K$, whose generic fibre is one dimensional and irreducible. These assumptions are enough to imply that each special fibre of $X$ is equidimensional of dimension one. Let $\eta$ be the generic point of $\Spec\roi_K$; closed points will be denoted by $s$, and we set $K_s=\Frac{\hat{\roi}_{K,s}}$, which is a local field of characteristic zero. Let $\Omega_{X/S}$ denote the coherent sheaf of relative differential (one-)forms.

Let $x\in X$ be a closed point, and $y\subset X$ a curve ($=$ one-dimensional, irreducible, closed subscheme) containing $x$; let $s$ be the closed point of $S$ under $x$. Then $A=\roi_{X,x}$ satisfies the conditions at the start of subsection \ref{subsection_reciprocity_for_incomplete_rings}, and contains the discrete valuation ring $\roi_{K,s}$. Also denote by $y\lhd \roi_{X,x}$ the local equation for $y$ at $x$; then $y$ is a height one prime of $A$, and we denote by \[\RRES_{x,y}:\Omega_{K(X)/K}\to K_s\] the residue map $\RRES_y:\Omega_{\Frac A/K}\to K_s$.

The following is the reciprocity law for residues around a fixed point on $X$:

\begin{theorem}\label{theorem_reciprocity_around_a_point}
Let $\omega\in\Omega_{K(X)/K}$, and let $x\in X$ be a closed point sitting over $s\in S$. Then for all but finitely many curves $y\subset X$ containing $x$, the residue $\RRES_{x,y}(\omega)$ is zero, and \[\sum_{\substack{y\subset X\\y\ni x}}\RRES_{x,y}(\omega)=0\] in the local field $K_s$.
\end{theorem}
\begin{proof}
This is simply the geometric statement of theorem \ref{theorem_incomplete_reciprocity}.
\end{proof}

\begin{remark}\label{remark_reciprocity_along_a_vertical_curve}
There is also a reciprocity law for the residue maps $(\RRES_{x,y})_x$ along any fixed vertical curve $y\subset X$. Let $\omega\in\Omega_{K(X)/K}$ and fix an irreducible component $y$ of a special fibre $X_s$. Then \[\sum_{x\in y}\RRES_{x,y}(\omega)=0\] in the local field $K_s$, where $x$ varies over all closed points of $y$. Remarkably, it is possible for infinitely many terms in the summation to be non-zero, but the sum does converge to zero in the valuation topology on $K_s$. This will be published in a later work focusing on Grothendieck duality\footnote{See \cite{Morrow2010}.}. Compare with \cite[Proposition 6]{Osipov2000}, where the same result for an algebraic variety over a smooth curve is established.
\end{remark}

\section{Explicit construction of the dualizing sheaf for arithmetic surfaces}\label{section_explicit_dualizing_sheaf_of_arith_surface}
Let $\pi:X\to S=\Spec\roi_K$ be an arithmetic surface in the sense of section \ref{section_reciprocity_for_surfaces}. In this section we will apply the theory of residues developed in sections \ref{section_local_relative_residues} and \ref{section_reciprocity_for_rings} to give an explicit description of the dualizing sheaf $\omegaa_\pi$ which is analogous to Yekutieli's result for algebraic varieties described in subsection \ref{subsection_explicit_grothendieck_duality}.

We begin by establishing an essential result in the affine setting, to which we will then reduce the main theorem \ref{main_theorem}.

\subsection{The affine case}
Let $\roi_K$ be a Dedekind domain of characteristic zero with finite residue fields; its field of fractions is $K$. We suppose that we are given a finitely generated, flat $\roi_K$-algebra $A$, which is normal and two-dimensional. Assume further that there is an intermediate ring $B$ \[\roi_K\le B\le A\] such that $B\cong\roi_K[T]$ and such that $A$ is a finitely generated $B$-module. Finally, set $F=\Frac A$, $M=\Frac B$, and note that $\Omega_{F/K}$ is a one-dimensional $F$-vector space, with basis $dT$.

If $0\lhd y\lhd x$ is a chain of primes in $A$, then $A_x$ is a two-dimensional, normal, local ring containing the discrete valuation ring $\roi_{K,s(x)}$, where $s(x)=\roi_K\cap x$. Therefore, as in section \ref{section_reciprocity_for_surfaces}, we have the residue map $\RRES_{x,y}:\Omega_{F/K}\to K_{s(x)}$ where $K_{s(x)}=\Frac\hat{\roi}_{K,s(x)}$. The situation is similar with $B$ in place of $A$.

We begin by establishing a functoriality result which we could have proved in section \ref{section_reciprocity_for_rings}:

\begin{proposition}\label{proposition_incomplete_functoriality}
Let $\omega\in\Omega_{F/K}$, and let $0\lhd y\lhd x\lhd B$ be a chain of primes in $B$. Then \[\RRES_{x,y}\Tr_{F/M}(\omega)=\sum_{x',y'}\RRES_{x',y'}(\omega)\] where the sum is taken over all chains $0\lhd y'\lhd x'\lhd A$ such that $x'$ sits over $x$ and $y'$ sits over $y$.
\end{proposition}
\begin{proof}
Let $x$ be a fixed maximal ideal of $B$; then $A\otimes_B\hat{B_x}=\bigoplus_{x'|x}\hat{A_{x'}}$ where $x'$ ranges over the finitely many maximal ideals of $A$ sitting over $x$. Therefore $F\otimes_M\Frac\hat{B_x}=\bigoplus_{x'|x}\Frac\hat{A_{x'}}$ and \[\Tr_{F/M}=\sum_{x'|x}\Tr_{\Frac\hat{A_{x'}}/\Frac\hat{B_x}},\] a local-global identity which is of course very well known for Dedekind domains.

Let $Y$ be a height one prime of $\hat{B_x}$. Then, for $\omega\in\Omega_{F/K}$,
\begin{align*}
\RRES_Y\Tr_{F/M}\omega
	&=\sum_{x'|x}\RRES_Y\Tr_{\Frac\hat{A_{x'}}/\Frac\hat{B_x}}\omega\\
	&\stackrel{(\ast)}{=}\sum_{x'|x}\sum_{Y'|Y}\RRES_{Y'}\omega
\end{align*}
where $Y'$ ranges over the height one primes of $\hat{A_{x'}}$ sitting over $Y$, and equality $(\ast)$ follows from  proposition \ref{proposition_functoriality_in_complete_case}. Now fix a height one prime $y$ of $B$ contained inside $x$; then
\begin{align*}
\RRES_{x,y}\Tr_{F/M}\omega
	&=\sum_{\substack{Y\lhd\hat{B_x}\\Y|y}}\RRES_Y\Tr_{F/M}\omega\\
	&=\sum_{\substack{Y\lhd\hat{B_x}\\Y|y}}\sum_{x'|x}\sum_{Y'|Y}\RRES_{Y'}\omega\\
	&=\sum_{x'|x}\sum_{\substack{y'\lhd A_{x'}\\y'|y}}\sum_{\substack{Y'\lhd\hat{A_{x'}}\\Y'|y'}}\sum_{Y'|Y}\RRES_{Y'}\omega\\
	&=\sum_{x'|x}\sum_{\substack{y'\lhd A_{x'}\\y'|y}}\RRES_{y'}\omega,
\end{align*}
which is the required result.
\end{proof}

We now introduce the following $A$-submodule of $\Omega_{F/K}$ defined in terms of residues: \[W_{A/\roi_K}=\{\omega\in\Omega_{F/K}:\RRES_{x,y}(f\omega)\in\hat{\roi}_{K,s(x)}\mbox{ for all } 0\lhd y\lhd x\lhd A \mbox{ and }f\in A_y\}.\] Similarly define $W_{B/\roi_K}$, and compare with Yekutieli's `holomorphic forms' in subsection \ref{subsection_explicit_grothendieck_duality}.

Suppose that $\omega\in W_{A/\roi_K}$ and $y\lhd x$ is a chain in $A$. We remarked at the end of the proof of theorem \ref{theorem_reciprocity_base_case} that each residue map on a two-dimensional, complete, normal local ring is continuous with respect to the adic topology on the ring and the discrete valuation topology on the local field (this is easy to prove for $\roi_K[[T]]$ and follows in the general case using functoriality). Therefore $\RRES_{x,y}(f\omega)\in\hat{\roi}_{K,s(x)}$ for all $f\in\hat{A_x}$. Another continuity argument even implies that this remains true for $f\in(\hat{A_x})_y$.

Now, $y\hat{A_x}$ is a radical ideal of $\hat{A_x}$; localising and completing with respect to this ideal obtains \[(\hat{(\hat{A_x})_y}=\bigoplus_{Y|y}\hat{(\hat{A_x})_Y}\] where $Y$ ranges over the height one primes of $\hat{A_x}$ sitting over $y$. Each $\roi_{x,Y}:=\hat{(\hat{A_x})_Y}$ is a complete discrete valuation ring whose field of fractions is a two-dimensional local field, which we will denote $F_{x,Y}$. Note that $\RRES_{x,y}=\sum_Y\RRES_{F_{x,Y}}$ by definition. 

Fix a particular height one prime $Y_0$ of $\hat{A_x}$ over $y$. Since $(\hat{A_x})_y$ is dense in $\bigoplus_{Y|y}\roi_{x,Y}$ with respect to the discrete valuation topologies, there is $h\in(\hat{A_x})_y$ which is $Y_0$-adically close to $1$ and $Y$-adically close to $0$ for $Y\neq Y_0$. More precisely, since each residue map $\RRES_{F_{x,Y}}$ is continuous with respect to the discrete valuation topologies on $F_{x,Y}$ and $K_{s(x)}$, we may take $h$ to satisfy
\begin{enumerate}
\item $\RRES_{F_{x,Y}}(h\roi_{x,Y}\omega)\subseteq\hat{\roi}_{K,s(x)}$ for $Y\neq Y_0$;
\item $\RRES_{F_{x,Y_0}}((h-1)\roi_{x,Y_0})\omega)\subseteq\hat{\roi}_{K,s(x)}$.
\end{enumerate}
Replacing $f$ by $hf$, it follows that $\RRES_{F_{x,Y_0}}(f\omega)\in\hat{\roi}_{K,s(x)}$ for all $f\in(\hat{A_x})_y$, and therefore for all $f\in\roi_{x,Y_0}$ by continuity. To summarise:

\begin{lemma}
Let $\omega\in\Omega_{F/K}$; then $\omega\in W_{A/\roi_K}$ if and only if $\RRES_{F_{x,Y}}(f\omega)\in\hat{\roi}_{K,s(x)}$ for all maximal ideals $x\lhd A$, all height one primes $Y\lhd\hat{A_x}$, and all $f\in\roi_{x,Y}$.
\end{lemma}
\begin{proof}
The implication $\Leftarrow$ is trivial, and we have just proved $\Rightarrow$.
\end{proof}

Next we reduce the calculation of $W_{A/\roi_K}$ to that of $W_{B/\roi_K}$:

\begin{lemma}
Let $\omega\in\Omega_{F/K}$; then $\omega\in W_{A/\roi_K}$ if and only if $\Tr_{F/M}(g\omega)\in W_{B/\roi_K}$ for all $g\in A$.
\end{lemma}
\begin{proof}
The implication $\Leftarrow$ follows from proposition \ref{proposition_incomplete_functoriality}. Let us fix a chain $y\lhd x$ in $B$ and suppose that $\RRES_{x,y}(f\Tr_{F/M}(g\omega))\in\hat{\roi}_{K,s(x)}$ for all $g\in A$, $f\in B_y$; so \begin{equation}\sum_{x',y'}\RRES_{x',y'}(g\omega)\in\hat{\roi}_{K,s(x)}\label{7}\end{equation} for all $g\in A_y$ by proposition \ref{proposition_incomplete_functoriality}. Since we have $\hat{A_x}=\oplus_{x'|x}\hat{A_{x'}}$, it follows that if $\xi$ is a fixed maximal ideal of $A$ over $x$, then there is $h\in A$ which is close to $1$ $\xi$-adically and close to $0$ $x'$-adically for any other maximal ideal $x'\neq\xi$ over $x$. More precisely, as we remarked at the end of the proof of theorem \ref{theorem_reciprocity_base_case}, each residue map on a two-dimensional, complete, normal local ring is continuous with respect to the adic topology on the ring and the discrete valuation topology on the local field (this is easy to prove for $\roi_K[[T]]$ and follows in the general case using functoriality); we may find $h$ such that
\begin{enumerate}
\item $\RRES_{x',y'}(hA_{x'}\omega)\subseteq\hat{\roi}_{K,s(x)}$ for $x'\neq\xi$ and $y'\lhd x'$ over $y$;
\item $\RRES_{\xi,y'}((h-1)A_{\xi}\omega)\subseteq\hat{\roi}_{K,s(x)}$ for $y'\lhd \xi$ over $y$.
\end{enumerate}
Replacing $g$ by $gh$ in (\ref{7}) obtains \[\sum_{\substack{y'\lhd\xi\\y'|y}}\RRES_{\xi,y'}(g\omega)\in\hat{\roi}_{K,s(x)}\] for all $g\in A$. This sum is equal to \[\sum_{\substack{y'\lhd\!^1\hat{A_\xi}\\y'|y}}\RRES_{y'}(g\omega),\] and we may now repeat the argument, similarly to how we proved the previous lemma, by completing at $y$ instead of $x$, and using the fact that the residue map on a two-dimensional local field is continuous with respect to the valuation topology. This gives $\RRES_{\xi,y'}(g\omega)\in\hat{\roi}_{K,s(x)}$ for all $g\in A_{y'}$, for any $y'\lhd\xi$ over $y$. This completes the proof.
\end{proof}

We may now establish our main result in the affine case, relating $W_{A/\roi_K}$ to the codifferent of $A/B$. A review of the codifferent and its importance in duality theory is provided in the appendix. The proof requires explicit arguments using residues, and uses the results and notation of sections \ref{section_local_relative_residues} and \ref{section_reciprocity_for_rings}.

\begin{theorem}\label{main_theorem_affine_case}
We have $W_{A/\roi_K}=\frak{C}(A/B)dT$.
\end{theorem}
\begin{proof}
Since $\Omega_{F/K}=F\,dT$ it is enough, by the previous lemma, to prove that $W_{B/\roi_K}=B\,dT$. Let $\omega=h\,dT\in\Omega_{M/K}$, where $h\in M$; we wish to prove $h\in B$. The following arguments only require that $B$ is smooth over $\roi_K$. Fix a maximal ideal $x\lhd B$ and write $s=s(x)$, $C=\hat{B_x}$, $N=\Frac C$ for simplicity; let $\pi\in\roi_{K,s}$ be a uniformiser at $s$.

If $y$ is a height one prime of $C$ which does not contain $\pi$, then $\pi^{-1}\in C_y$ and so \[\RRES_y(f\omega)\in\hat{\roi}_{K,s}\mbox{ for all }f\in \hat{C_y}\Longleftrightarrow\RRES_y(f\omega)=0\mbox{ for all }f\in\hat{C_y}.\] Note that in the notation earlier in this section, $\hat{C_y}=\roi_{x,y}$. Furthermore, non-degeneracy of the trace map from the coefficient field $k_y$ to $K_s$ implies \[\RRES_y(f\omega)=0\mbox{ for all }f\in\hat{C_y}\Longleftrightarrow\RES_y(f\omega)=0\mbox{ for all }f\in \hat{C_y}.\] Let $t\in C_y$ be a uniformiser at $y$; then $\omega=h\frac{dT}{dt}\,dt$ and it easily follows from the definition of the residue map on the equi-characteristic two-dimensional field $N_{y\hat{C}}=M_{x,y}\cong k_y((t))$ that \[\RES_y(f\omega)=0\mbox{ for all }f\in\hat{C_y}\Longleftrightarrow h\frac{dT}{dt}\in \hat{C_y}.\]  Finally, we have identifications \[\hat{C_y}\,dT=\Omega_{B/\roi_K}\otimes_B\hat{C_y}\cong\Omega_{\hat{C_y}/\hat{\roi}_{K,s}}^\sub{sep}=\hat{C_y}\,dt,\] with the isomorphism coming from lemma \ref{lemma_differential_forms_global_to_local}, and $dT$ corresponding to $\frac{dT}{dt}\,dt$. Hence $\frac{dT}{dt}$ is a unit in $\hat{C_y}$, and so \[\RES_y(f\omega)=0\mbox{ for all }f\in\hat{C_y}\Longleftrightarrow h\in C_y.\]

Now we consider the prime(s) containing $\pi$. The special fibre $B/\pi B$ is smooth, and so $C/\pi C$ is a complete, regular, one-dimensional local ring, i.e. a complete discrete valuation ring, and $\pi C$ is prime in $C$. Therefore $\pi C$ is the only height one prime of $C$ which contains $\pi$. Furthermore, $\pi$ is a uniformiser in the two-dimensional local field $N_{\pi A}=M_{x,\pi A}$, and therefore by corollary \ref{corollary_structure_of_standard_fields} there is an isomorphism $F_{\pi C}\cong k_{\pi C}\dblecurly{t}$, and moreover $k_{\pi C}$ is an unramified extension of $K_s$. It easily follows from the definition of the residue map in this case that \[\RES_{\pi C}(f\omega)\in\roi_{k_{\pi\hat{C}}}\mbox{ for all }f\in\hat{C_{\pi C}}\Longleftrightarrow h\in C_{\pi C}.\] The fact that the extension $k_{\pi C}/K$ of local fields is unramified now implies \[\RRES_{\pi C}(f\omega)\in\roi_{k_{\pi\hat{C}}}\mbox{ for all }f\in\hat{C_{\pi C}}\Longleftrightarrow h\in C_{\pi C}.\]

We conclude that \[\RRES_{M_{x,y}}(f\omega)\in\hat{\roi}_{K,s}\mbox{ for all }y\lhd\!^1\hat{B_x}\mbox{ and }f\in\roi_{x,y}\Longleftrightarrow h\in (\hat{B_x})_y \mbox{ for all }y\lhd\!^1\hat{B_x}.\] But $\hat{B_x}$ is normal, so $\bigcap_{y\lhd\!^1\hat{B_x}}(\hat{B_x})_y=\hat{B_x}$. We deduce that $\omega$ belongs to $W_{B/\roi_K}$ if and only if $h\in B_x$ for all $x$, which holds if and only if $h\in B$. This completes the proof.
\end{proof}

\subsection{The main result}
With the required calculation in the affine case completed, we now turn to the proof of the main theorem of this section. Let $\pi:X\to S=\Spec\roi_K$ be an arithmetic surface in the sense of section \ref{section_reciprocity_for_surfaces}. The following result is surely known, but the author could not find any reference to it:

\begin{lemma}\label{lemma_normalization_result}
Each point of $S$ has an affine open neighbourhood $U$ for which there exists a finite morphism of $S$-schemes $\rho:\pi^{-1}(U)\to\bb{P}_U^1$.
\end{lemma}
\begin{proof}
Let $s\in S$ be a point of the base, and write \[X^s:=X\times_S\Spec\roi_{S,s}=\operatorname{Proj}\roi_{S,s}[T_0,\dots,T_r]/J\] for some homogeneous ideal $J\lhd \roi_{S,s}[T_0,\dots,T_r]$. Therefore the fibre over $s$ is the projective variety $X_s$ over the field $k:=k(s)$ given by $X_s=\operatorname{Proj}k[T_0,\dots,T_r]/\res{J}$, where $\res{J}$ is the image of $J$ in $k[T_0,\dots,T_r]$. Let $\frak{p}_1,\dots,\frak{p}_m$ be the homogeneous prime ideals of $k[T_0,\dots,T_r]$ which are minimal over $\res{J}$, so that $\{V(\frak{p}_i)\}_i$ are the irreducible components of $X_s$. By the homogeneous prime avoidance lemma \cite[Lemma 3.3]{Eisenbud1995}, there is a homogeneous polynomial $g_0\in \langle T_0,\dots,T_r\rangle$ which is not contained in any of $\frak{p}_1,\dots,\frak{p}_m$.

Now, if $\frak{m}$ is a homogeneous prime ideal of $k[T_0,\dots,T_r]$ containing $J$ and $g_0$, then it strictly contains $\frak{p}_i$ for some $i$. Therefore $V(\frak{m})$ is strictly contained inside the irreducible curve $V(\frak{p}_i)$, whence $V(\frak{m})$ is a closed point of $X_s$. Another application of prime avoidance provides us with a homogeneous polynomial $g_1\in\langle T_0,\dots,T_r\rangle$ which is not contained inside any of the minimal homogeneous primes over $J$ and $g_0$; in other words, $V(g_1)$ does not contain any of the irreducible components of $V(g_0)\cap X_s$. Repeating the dimension argument again we see that there are no homogeneous prime ideals of $k[T_0,\dots,T_r]$ containing all of $J$, $g_0$, and $g_1$ (except for $\langle T_0,\dots,T_r\rangle$, of course); in other words, $V(g_1)\cap V(g_0)\cap X_s=\emptyset$. After replacing $g_0,g_1$ with $g_0^{\deg g_1},g_1^{\deg g_0}$ respectively, we may also assume that $\deg g_0=\deg g_1$.

The homomorphism of graded rings \[k[T_0,T_1]\to k[T_0,\dots,T_r]/\res{J},\quad T_i\mapsto g_i\quad (i=0,1)\] is injective because otherwise there would be a non-trivial relation between $g_0$ and $g_1$ of the form $h(g_0,g_1)\in \res{J}$, for some non-zero polynomial $h$; this would contradict the choice of $g_0,g_1$. Therefore there is an induced surjection \[\phi:X_s\setminus V\to \bb{P}_k^1,\] where $V=V(g_1)\cap V(g_0)\cap X_s$. But by choice of $g_0,g_1$ we know that $V=\emptyset$. Moreover, $\phi$ has finite fibres by dimension considerations, and hence, being projective, it is a finite morphism (by \cite[Proposition 4.4.2]{EGA_III_I}).

More importantly, we may also consider \[\roi_{S,s}[T_0,T_1]\to \roi_{S,s}[T_0,\dots,T_r]/J,\quad T_i\mapsto \tilde{g}_i\quad (i=0,1),\] where $\tilde{g}_i$ are lifts of $g_i$ to $\roi_{S,s}[T_0,\dots,T_r]$ (we chose the lifts to be homogeneous and of the same degree). This induces $\tilde{\phi}:X^s\setminus\tilde{V}\to\bb{P}_{\roi_{S,s}}^1$, where $\tilde{V}$ is the closed subscheme of $X^s$ defined by $\tilde{g}_0$ and $\tilde{g}_1$. We claim that $\tilde{V}=\emptyset$. Indeed, $\pi(\tilde{V})$ is a closed subset of $\Spec{\roi_{S,s}}$ which does not meet the closed point since $\tilde{V}\cap X_s=V=\emptyset$; this is only possible if $\pi(\tilde{V})$ is empty. Therefore $\tilde{\phi}$ is everywhere defined, giving a projective morphism $\tilde{\phi}:X^s\to\bb{P}_{\roi_{S,s}}$ whose restriction to the fibre $X_s$ is exactly $\phi$.

Next, extend $\tilde{\phi}$ to a morphism $\rho:\pi^{-1}(V)\to\bb{P}_V^1$, where $V\subseteq S$ is an open neighbourhood of $s$. According to A.~Grothendieck's interpretation \cite[Proposition 4.4.1]{EGA_III_I} of Zariski's main theorem, the set \[\pi^{-1}(V)_\rho:=\{x\in \pi^{-1}(V):x\mbox{ is an isolated point of the fibre } \rho^{-1}(\rho(x))\}\] is open in $\pi^{-1}(V)$; moreover, it contains all of $X_s$. Since $\pi$ is a closed mapping, it follows now that \[U:=V\setminus \pi(\pi^{-1}(V)\setminus \pi^{-1}(V)_\rho)\] is an open neighbourhood of $s$, contained inside $V$. By construction $\pi^{-1}(U)\subseteq \pi^{-1}(V)_\rho$, and so the restriction of $\rho$ to $\pi^{-1}(U)$ gives a projective morphism to $\bb{P}_U^1$ with everywhere finite fibres; hence $\rho|_{\pi^{-1}(U)}:\pi^{-1}(U)\to\bb{P}_U^1$ is a finite morphism by \cite[Proposition 4.4.2]{EGA_III_I}. We are free to replace $U$ by a smaller affine neighbourhood of $s$ if we wish, and this completes the proof.
\end{proof}

Given a morphism $f$ between suitable schemes, we will write $\omegaa_f$ for its dualizing sheaf (see subsection \ref{subsection_explicit_grothendieck_duality} for a reminder). Let $x\in X$, and pick an open affine neighbourhood $U=\Spec R\subseteq S$ of $\pi(s)$ and a finite morphism $\rho:\pi^{-1}(U)\to\bb{P}_R^1$, as provided by the previous lemma. Then $\bb{P}_R^1$ is covered by two open sets, each isomorphic to $\bb{A}_R^1$, and one of these open sets contains $\rho(x)$. Taking the preimage of this affine open set under $\rho$ gives us as affine open neighbourhood $\Spec A$ of $x$ and a finite injection $R[T]\into A$. We claim:

\begin{lemma}\label{lemma_explicit_formula_using_conductor}
$\omegaa_\pi|_{\Spec A}$, the restriction of the dualizing sheaf of $\pi$ to the open neighbourhood $\Spec A$ of $x$, may be naturally identified with the $A$-module \[\frak{C}(A/R[T])dT\subseteq\Omega_{A/R}.\]
\end{lemma}
\begin{proof}
This is an exercise in the functoriality properties of dualizing sheaves. Firstly, with $U=\Spec R$ as above, the restriction of $\omegaa_\pi$ to $\pi^{-1}(U)$ is the dualizing sheaf of the restriction of $\pi$ to $\pi^{-1}(U)$; or in symbols, \[(\omegaa_\pi)|_{\pi^{-1}(U)}=\omegaa_{\pi'}\] with $\pi'=\pi|_{\pi^{-1}(U)}$. This is immediate from the definition of a dualizing sheaf. So to simplify notation, we may replace $S$ by $\Spec R$ and assume that we have a finite morphism $\rho:X\to\bb{P}_R^1$.

Secondly, the dualizing sheaf of $\bb{P}_R^1\to\Spec R$ is equal to its canonical sheaf $\Omega_{\bb{P}_R^1/R}$ (e.g. \cite[Proposition 6.4.22]{Liu2002}). Since $\Omega_{\bb{P}_R^1/R}$ is locally free, it follows from \cite[Proposition 6.4.26]{Liu2002} that \[\omegaa_\pi=\omegaa_\rho\otimes_{\roi_X}\rho^*\Omega_{\bb{P}_R^1/R}.\] With $A$ and $R[T]$ as above, now restrict $\omegaa_\pi$ to $\Spec A$ to get 
\[\omegaa_\pi|_{\Spec A}=\omegaa_\rho|_{\Spec A}\otimes_{\roi_{\Spec A}}\rho^*\Omega_{\bb{A}_R^1/R}.\] Moreover, $\omegaa_\rho|_{\Spec A}$ is the dualizing sheaf of the finite morphism $\Spec A\to\Spec R[T]$, and this corresponds to the codifferent $\frak{C}(A/R[T])$ according to the discussion at the start of the appendix. Finally, $\rho^*\Omega_{\bb{A}_R^1/R}$ corresponds to the $A$-module $AdT=\Omega_{R[T]/R}\otimes_{R[T]} A$, completing the proof.
\end{proof}

All results are now in place to allow us to easily give the explicit representation of the dualizing sheaf $\omegaa_\pi$ of our arithmetic surface in terms of residues:

\begin{theorem}\label{main_theorem}
The dualizing sheaf $\omegaa_\pi$ of $X\to S$ is explicitly given by, for open $U\subseteq X$, \begin{align*}\omegaa_\pi(U)=\{\omega\in\Omega_{K(X)/K}:&\RRES_{x,y}(f\omega)\in\hat{\roi}_{K,\pi(x)}\mbox{ for}\\& \mbox{all } x\in y\subset U\mbox{ and }f\in\roi_{X,y}\}\end{align*} where $x$ runs over all closed points of $X$ inside $U$ and $y$ runs over all curves containing $x$.
\end{theorem}
\begin{proof}
This reduces to the affine situation of $\Spec A\subseteq X$ as above, i.e. we have a finite morphism $R[T]\into A$, with $\Spec R$ being an affine open subset contained in $\pi(\Spec A)\subseteq S$. The theorem now follows immediately from the previous lemma combined with theorem \ref{main_theorem_affine_case}, replacing $R$ by $\roi_K$ and $R[T]$ by $B$.
\end{proof}

\begin{appendix}
\section{Finite morphisms, differents and Jacobians}\label{section_finite_morphisms_etc}
The purpose of this appendix is to prove some results pertaining to the relationship between the canonical sheaf and dualizing sheaf, especially in the case of finite morphisms. In particular, we prove that for a finite morphism which is a local complete intersection, the canonical and dualizing sheaves coincide. From this we deduce the same equality for an arithmetic surface (and the proof works in far greater generality) which is a local complete intersection. This equality for local complete intersections is well known in duality theory (e.g. \cite[Corollary 19 and Proposition 22]{Kleiman1980}, but our proof is based on more explicit local calculations. Still, this material may be known to experts and the author would not be surprised to learn if a comprehensive discussion is buried somewhere in SGA or EGA\footnote{J.-P.~Serre gave a talk at Harvard's `Basic Notions' seminar, 10 November 2003, entitled ``Writing Mathematics?'', in which he explains how to write mathematics badly. He explains that if you wish to give a reference which can not be checked by the reader, then you should ideally refer, without any page references, to the complete works of Euler, but ``if you refer to SGA or EGA, you have a good chance also''. The reader interested in verifying {\em this} reference should consult timeframe 4.11--4.20 of the video at \url{http://modular.fas.harvard.edu/edu/basic/serre/}.}.

These calculations are not necessary for the main part of the paper, but are similar in spirit (and were required in an earlier version of the paper), hence have been relegated to an appendix.

Suppose that $A/B$ is a finite extension of normal, excellent rings, with corresponding fraction fields $F/M$ (assumed to be of characteristic zero to avoid any inseparability problems). The associated {\em codifferent} is the $A$-module \[\frak{C}(A/B)=\{x\in F:\Tr_{F/M}(xA)\subseteq B\}.\] If $P$ is any $A$-module, then the natural pairing \[\Hom_A(P,\frak{C}(A/B))\times P\To\frak{C}(A/B)\stackrel{\Tr_{F/M}}{\To}B\] induces a $B$-linear map \[\Hom_A(P,\frak{C}(A/B))\to\Hom_B(P,B),\] which is easily checked to be an isomorphism (using non-degeneracy of $\Tr_{F/M}$). Thus $\frak{C}(A/B)$ is exactly the Grothendieck dualizing module of the projective morphism $\Spec A\to\Spec B$. Note also that we have a natural identification (taking $P=A$): \[\frak{C}(A/B)=\Hom_B(A,B).\]

We will prove that when $A/B$ is a local complete intersection, then the codifferent is an invertible $A$-module generated (locally) by the determinants of certain Jacobian matrices:

\begin{theorem}\label{theorem_different_and_jacobian_in_general}
Assume $A/B$ is a local complete intersection. Then the codifferent $\frak{C}(A/B)$ is an invertible $A$-module (i.e. locally free of rank $1$). More precisely, if $\frak{m}$ is a maximal ideal of $A$ and we write \[A_{\frak{m}}=B[T_1,\dots,T_m]_{\frak{n}}/\langle f_1,\dots,f_m\rangle,\] where $f_1,\dots,f_m$ is a regular sequence inside a maximal ideal $\frak{n}\lhd B[T_1,\dots,T_m]$ sitting over $\frak{m}_B$, then \[\frak{C}(A/B)_{\frak{m}}=J_\frak{m}^{-1}A_{\frak{m}},\] where $J_\frak{m}:=\det\left(\dd{f_i}{T_j}\right)_{i,j}\in A_\frak{m}$.
\end{theorem}

Before the proof, we collect together some corollaries:

\begin{corollary}\label{corollary_dualizing_is_canonical_for_finite_morphism}
Assume $A/B$ is a local complete intersection. Then the dualizing module $\frak{C}(A/B)$ is naturally isomorphic to the canonical module $\omega_{A/B}$.
\end{corollary}
\begin{proof}
As soon as we recall the definition and main properties of the canonical module this will be clear. For simplicity we initially assume that $A/B$ is a global complete intersection of the form \[A=B[T_1,\dots,T_m]/I=B[\underline{T}]/I,\] where $I$ is an ideal generated by a regular sequence $f_1,\dots,f_m$. In general we will only have such a representation of $A/B$ locally.

The $A$-module $I/I^2$ is free of rank $m$, with basis $f_1,\dots,f_m$ (or rather, the images of these mod $I^2$), and there is a natural exact sequence of $A$-modules \[I/I^2\stackrel{\delta}{\to}\Omega_{B[\underline{T}]/B}\otimes_{B[\underline{T}]}A\to\Omega_{A/B}\to0,\] where $\delta$ is given by \[\delta:I/I^2\to\Omega_{B[\underline{T}]/B}\otimes_{B[\underline{T}]}A,\quad b\mbox{ mod}{I}\mapsto db.\] Since $\delta(f_i)=\sum_{j=1}^m\dd{f_i}{T_j}dT_j$, the matrix of $\delta$ with respect to the bases $f_1,\dots,f_m$ and $dT_1,\dots,dT_m$ is the Jacobian matrix \[\cal{J}:=\left(\dd{f_i}{T_j}\right)_{\substack{1\le i\le m\\1\le j\le m}}\] (more precisely, the image of this matrix in $A$). Since $L/F$ is separable the Jacobian condition for smoothness asserts that $\mbox{rank}\,\cal{J}=m$. That is, $\cal{J}$ is invertible (in $GL_m(F)$) and so $\delta$ is actually injective: \[0\to I/I^2\stackrel{\delta}{\to}\Omega_{B[\underline{T}]/B}\otimes_{B[\underline{T}]}A\to\Omega_{A/B}\to0.\] Put $J=\det\cal{J}\neq 0$.

The {\em relative canonical module} $\omegaa_{A/B}$ is the invertible $A$-module \begin{align*}\omegaa_{A/B}&=\Hom_A(\det I/I^2,A)\otimes_A\det(\Omega_{B[\underline{T}]/B}\otimes_{B[\underline{T}]}A)\\&=\Hom_A(\det I/I^2,\det(\Omega_{B[\underline{T}]/B}\otimes_{B[\underline{T}]}A))\end{align*} where $\det I/I^2=\bigwedge_A^mI/I^2$ and $\det(\Omega_{B[\underline{T}]/B}\otimes_{B[\underline{T}]}A)=\bigwedge_A^m(\Omega_{B[\underline{T}]/B}\otimes_{B[\underline{T}]}A)$. The $A$-module $\omegaa_{A/B}$ is generated by the morphism $\ep$ characterised by $\ep(f_1\wedge\dots\wedge f_m)=dT_1\wedge\dots\wedge dT_m$, but it also contains a canonical element $\delta^{\wedge m}\in\omegaa_{A/B}$. Moreover, since \[\delta^{\wedge m}(f_1\wedge\dots\wedge f_m)=J dT_1\wedge\dots\wedge dT_m,\] we have the relation $\delta^{\wedge m}=J\ep$. In other words, there is a canonical isomorphism \[\omegaa_{A/B}\cong J^{-1}A,\quad a\delta^{\wedge m}\mapsto a.\] For more details, see \cite[\S6.4]{Liu2002}.

When $A$ cannot necessarily be written as a global complete intersection over $A$, then $\omegaa_{A/B}$ is defined by patching the previous construction on sufficiently small open subsets of $\Spec A$. The isomorphism just described is sufficiently canonical to patch together to give an embedding $\omegaa_{A/B}\into F$. If $\frak{m}$ is a maximal ideal of $A$ and we write $A_{\frak{m}}$ as a complete intersection over $A$, then under this identification we have $\omegaa_{A/B}A_\frak{m}=J_\frak{m}^{-1}A_{\frak{m}}$, where $J_\frak{m}$ is the determinant of the Jacobian matrix corresponding to the complete intersection. The previous theorem implies that this is nothing other than $\frak{C}(A/B)$.
\end{proof}

\begin{corollary}
Let $\pi:X\to S$ be an arithmetic surface in the sense of section \ref{section_reciprocity_for_surfaces}. Then the dualizing sheaf $\omegaa_\pi$ and relative canonical sheaf $\omegaa_{X/S}$ are isomorphic.
\end{corollary}
\begin{proof}
Using lemma \ref{lemma_normalization_result} we may suppose that we have a finite morphism $\rho:X\to\bb{P}_S^1$. Functoriality of the canonical sheaf implies that $\omegaa_{X/S}\cong \omegaa_{X/\bb{P}_S^1}\otimes_{\roi_X} \rho^*\Omega_{\bb{P}_S^1/S}$ (see e.g. \cite[Theorem 6.4.9a]{Liu2002}). Since $\pi$ is a local complete intersection, so is $\rho$. Indeed, we may pick a section $\sigma:S\to\bb{P}_S^1$ which is a regular embedding, and then any regular immersion $U\to \bb{A}_S^r$, with $U$ an open subset of $X$, induces a regular embedding \[U\to\bb{A}_S^r=\bb{A}_S\times_SS\to\bb{A}_S^r\times_S\bb{P}_S^1\] (since regular embeddings are closed under composition and flat base change); this gives a regular embedding of $U$ into a smooth $\bb{P}_S^1$-scheme, as required.

The previous corollary now implies that the canonical sheaf $\omegaa_{X/\bb{P}_S^1}$ coincides with the dualizing sheaf $\omegaa_\rho$. But we already saw in lemma \ref{lemma_explicit_formula_using_conductor} that $\omegaa_\pi\cong \omegaa_\rho\otimes_{\roi_X} \rho^*\Omega_{\bb{P}_S^1/S}$, which completes the proof.
\end{proof}

\begin{remark}
The previous result requires little more than combining lemma \ref{lemma_explicit_formula_using_conductor} with corollary \ref{corollary_dualizing_is_canonical_for_finite_morphism} and some basic results on the functoriality of canonical and dualizing sheaves. It seems to hold in the generality of having a projective local complete intersection between normal, excellent schemes, and thus provides a proof in this case that the canonical and dualizing sheaves coincide without resorting to the methods of higher duality theory used in \cite{Kleiman1980}.
\end{remark}

\subsection{The case of complete discrete valuation rings}
To prove theorem \ref{theorem_different_and_jacobian_in_general}, we begin by treating the case of complete discrete valuation rings. Let $F/M$ be a finite, separable extension of complete discrete valuation fields, with rings of integers $\roi_F/\roi_M$. In place of the codifferent, one usually considers the {\em different} $\frak{D}(\roi_F/\roi_M)$, which is the $\roi_F$-fractional ideal defined by \[\frak{C}(\roi_F/\roi_M)\,\frak{D}(\roi_F/\roi_M)=\roi_F\] i.e. the complement of the codifferent. Since $\roi_F/\roi_M$ is a finite extension of regular, local rings, it is a complete intersection \[\roi_F=\roi_M[T_1,\dots,T_m]_\frak{n}/\langle f_1,\dots,f_m\rangle,\] where $f_1,\dots,f_m$ is a regular sequence contained inside a maximal ideal $\frak{n}\lhd \roi_M[T_1,\dots,T_m]$ which sits over the maximal ideal of $\roi_M$. We set \[\frak{J}(\roi_F/\roi_M)=\det\left(\dd{f_i}{T_j}\right)_{i,j}\roi_F,\] which we may as well call the {\em Jacobian ideal}. The fact that $F/M$ is separable implies that the Jacobian ideal is non-zero, and as argued in corollary \ref{corollary_dualizing_is_canonical_for_finite_morphism}, we have an exact sequence \[0\to \langle f_1,\dots,f_m\rangle/\langle f_1,\dots,f_m\rangle^2\stackrel{\delta}{\to}\Omega_{\roi_M[\underline{T}]/\roi_M}\otimes_{\roi_M[\underline{T}]}\roi_F\to\Omega_{\roi_F/\roi_M}\to0.\] The matrix of $\delta$ with respect to the bases $f_1,\dots,f_m$ and $dT_1,\dots,dT_m$ is the Jacobian matrix, and it easily follows (using the Iwasawa decomposition of $GL_m(F)$) that $\frak{J}(\roi_F/\roi_M)=\frak{p}_F^l$, where $l=\mbox{length}_{\roi_F}\Omega_{\roi_F/\roi_M}$; in particular, the Jacobian ideal does not depend on how we write $\roi_F$ as a complete intersection over $\roi_M$.

We are going to prove that \begin{equation}\frak{J}(\roi_F/\roi_M)=\frak{D}(\roi_F/\roi_M).\label{8}\end{equation} When $F/M$ is monogenic (i.e. we may write $\roi_F=\roi_M[\al]$ for some $\al\in\roi_F$), which is the case whenever the residue field extension of $F/M$ is separable, the equality (\ref{8}) is well-known; it states that $\frak{D}(\roi_F/\roi_M)=g'(\al)$, where $g$ is the minimal polynomial of $\al$ over $M$. A proof may be found in \cite[III.2]{Neukirch1999} (this reference assumes throughout that the residue field extensions are separable, but the proof remains valid in the general case).

Several easy lemmas are required, firstly a product formula:

\begin{lemma}
Let $F'$ be a finite, separable extension of $F$; then \[\frak{D}(\roi_{F'}/\roi_M)=\frak{D}(\roi_F/\roi_M)\frak{D}(\roi_{F'}/\roi_F)\] and \[\frak{J}(\roi_{F'}/\roi_M)=\frak{J}(\roi_F/\roi_M)\frak{J}(\roi_{F'}/\roi_F).\]
\end{lemma}
\begin{proof}
The different result is well-known; see e.g. \cite[III.2]{Neukirch1999}. We will prove the Jacobian result. Write $\roi_{F'}$ as a complete intersection over $\roi_F$: \[\roi_{F'}=\roi_F[T_{m+1},\dots,T_{m+n}]_{\frak{n}'}/\langle f_{m+1},\dots,f_{m+n}\rangle\] (so here $f_{m+1},\dots,f_{m+n}$ is a regular sequence in a maximal ideal $\frak{n}'\lhd\roi_F[T_{m+1},\dots,T_{m+n}]$ which sits over the maximal ideal of $\roi_F$), and denote by $\tilde{f_i}$ a lift of the $\roi_F$ polynomials $f_i$ to $\roi_M[T_1,\dots,T_{m+n}]$, for $i=m+1,\dots,m+n$. Then \[\roi_{F'}=\roi_M[T_1,\dots,T_{m+n}]_{\frak{n}''}/\langle f_1,\dots,f_m,\tilde{f}_{m+1},\dots,\tilde{f}_{m+n}\rangle\] represents $\roi_{F'}$ as a complete intersection over $\roi_M$; here $\frak{n}''$ is the pull-back of $\frak{n}'$ via $\roi_M[T_1,\dots,T_{m+n}]\to\roi_F[T_{m+1},\dots,T_{m+n}]$. The Jacobian matrix in $\roi_{F'}$ associated to this complete intersection is
\[\left(\begin{array}{ll}
\left(\dd{f_i}{T_j}\right)_{i,j=1,\dots,m} & \quad\quad\quad0\\
\left(\dd{f_i}{T_j}\right)_{\substack{i=m+1,\dots,m+n\\j=1,\dots,m}} & \left(\dd{f_i}{T_j}\right)_{i,j=m+1,\dots,m+n}\\
\end{array}\right).\] Since the determinant of this is the product of the determinants of the two square matrices, the proof is complete.
\end{proof}

\begin{lemma}
Suppose additionally that $F/M$ is Galois. Then there exists a sequence of intermediate extensions $F=F_s>\dots>F_{-1}=M$ such that each extension $F_i/F_{i-1}$ is monogenic.
\end{lemma}
\begin{proof}
Let $F_0$ denote the maximal unramified subextension of $M$ inside $F$, and $F_1$ the maximal tamely ramified subextension (and set $F_{-1}=M$). Then $F/F_1$ is an extension whose residue field extension is purely inseparable, and whose ramification degree is a power of $p$ ($=$ the residue characteristic, which we assume is $>0$, for else $F_1=F$); therefore $\Gal(F/F_1)$ is a $p$-group, hence nilpotent, and so there is a sequence of intermediate fields $F=F_m>\dots>F_1$ such that each $F_i$ is a normal extension of $F_1$ and such that each step is a degree $p$ extension.

Then $\roi_{F_0}=\roi_{F_{-1}}[\theta]$ where $\theta\in\roi_{F_0}$ is a lift of a generator of $\res{F}_0/\res{M}$. Also, $\roi_{F_1}=\roi_{F_0}[\pi]$ where $\pi$ is a uniformiser of $F_1$. It remains to observe that any extension of prime degree $F_i/F_{i-1}$ is monogenic. Indeed, it is either totally ramified in which case $\roi_{F_i}=\roi_{F_{i-1}}[\pi']$ where $\pi'$ is a uniformiser of $F_i$; or else the ramification degree is $1$ and $\roi_{F_i}=\roi_{F_{i-1}}[\theta']$ where $\theta'\in\roi_{F_i}$ is a lift of a generator of the degree $p$ extension $\res{F}_i/\res{F}_{i-1}$ (which may be inseparable).
\end{proof}

Combining the previous two lemmas with the validity of (\ref{8}) in the monogenic case, we have proved that (\ref{8}) holds for any finite, Galois extension $F/M$. Now suppose that $F/M$ is finite and separable, but not necessarily normal, and let $F'$ be the normal closure of $F$ over $M$. The product formula gives us \[\nu_{F'}(\frak{D}(\roi_{F'}/\roi_M))=e_{F'/F}\nu_F(\frak{D}(\roi_F/\roi_M))+\nu_{F'}(\frak{D}(\roi_{F'}/\roi_F)),\] and similarly for $\frak{J}$. But the Galois case implies that (\ref{8}) is true for $F'/M$ and $F'/F$. We deduce $\frak{J}(\roi_F/\roi_M)=\frak{D}(\roi_F/\roi_M)$, which establishes our desired result. To summarise:

\begin{theorem}\label{theorem_different_and_jacobian_for_cdvr}
Let $F/M$ be a finite, separable extension of complete discrete valuation fields. Write $\roi_F$ as a complete intersection over $\roi_M$ as above, and let $J\in \roi_F$ be the determinant of the Jacobian matrix. Then $J\neq 0$ and \[\frak{C}(\roi_F/\roi_M)=J^{-1}\roi_F.\]
\end{theorem}
\begin{proof}
Replacing $\frak{C}(\roi_F/\roi_M)$ by its complementary ideal $\frak{D}(\roi_F/\roi_M)$, this is what we have just proved.
\end{proof}

The previous theorem is an elementary result concerning the ramification theory of complete discrete valuation fields with imperfect residue fields \cite{Abbes2002} \cite{Abbes2003} \cite{Xiao2007} \cite{Xiao2008a} \cite{Xiao2008b} \cite{Zhukov2000} \cite{Zhukov2003}.

\subsection{The higher dimensional case}
We now generalise, in several stages, from complete discrete valuation rings to the general case. Let $A/B$ be a finite extension of normal, excellent rings as at the start of the appendix, with $A/B$ a local complete intersection. For a height one prime $y\lhd B$, the localisation $B_y$ is a discrete valuation ring, and we set $M_y=\Frac\hat{B_y}$; use similar notation for $A$.

To proceed further, we need the following result, which I learned from \cite[exercise 6.3.5]{Liu2002}:

\begin{lemma}\label{lemma_flatness_of_finite_ci}
$A$ is flat over $B$.
\end{lemma}
\begin{proof}
Let $\frak{q}\lhd A$ be a maximal ideal of $A$, and set $\frak{p}=B\cap\frak{q}$. Since $A$ is a local complete intersection over $B$, we may therefore write \[A_{\frak{q}}=B_{\frak{p}}[T_1,\dots,T_m]_{\frak{n}}/\langle f_1,\dots,f_m\rangle,\] where $\frak{n}$ is a maximal ideal of $B[T_1,\dots,T_m]$ which sits over $\frak{p}$, and where $f_1,\dots,f_m\in B_{\frak{p}}[T_1,\dots,T_m]$ is a regular sequence inside $\frak{n}$. We will also write $\frak{n}$ for the image of $\frak{n}$ in $k(\frak{p})[T_1,\dots,T_m]$.

We will first prove that (the images of) $f_1,\dots,f_m$ form a regular sequence in $k(\frak{p})[T_1,\dots,T_m]_{\frak{n}}$. Well, if they do not, then pick $s$ minimally so that $f_s$ is a zero divisor in \[k(\frak{p})[T_1,\dots,T_m]_{\frak{n}}/\langle f_1,\dots,f_{s-1}\rangle.\] This latter ring (call it $R$) is the quotient of a regular, local ring by a regular sequence (by minimality of $s$), and hence is Cohen-Macaulay \cite[\S21]{Matsumura1989}. Any Cohen-Macaulay local ring contains no embedded primes (and so the zero-divisor $f_s$ belongs to a minimal prime of $R$) and is equi-dimensional \cite[Corollaries 18.10 and 18.11]{Eisenbud1995}; together these imply that $\dim R=\dim R/\langle f_s\rangle$. Quotienting out by any other $f_i$ drops the dimension by at most one (by Krull's principal ideal theorem), so we deduce \[\dim k(\frak{p})[T_1,\dots,T_m]_{\frak{n}}/\langle f_1,\dots,f_m\rangle\ge\dim k(\frak{p})[T_1,\dots,T_m]_{\frak{n}}-(m-1).\] But the ring on the left is a localisation of the fibre $A\otimes_B k(\frak{p})$, which is a finite dimensional $k(\frak{p})$-algebra, and so is zero-dimensional. Hence $\dim k(\frak{p})[T_1,\dots,T_m]_{\frak{n}}\le m-1$, contradicting the fact that ${\frak{n}}$ is a maximal ideal of $k(\frak{p})[T_1,\dots,T_m]$.

Secondly, since $B_{\frak{p}}\to B[T_1,\dots,T_m]_{\frak{n}}$ is a flat map of local rings, and $f_1$ is not a zero-divisor in $k(\frak{p})[T_1,\dots,T_m]_{\frak{n}}$, a standard criterion (e.g. \cite[Theorem 22.5]{Matsumura1989}) implies that $B_{\frak{p}}\to B[T_1,\dots,T_m]_{\frak{n}}/\langle f_1 \rangle$ is flat. Applying this criterion another $m-1$ times, we deduce that $B_{\frak{p}}\to B[T_1,\dots,T_m]_{\frak{n}}/\langle f_1,\dots,f_m \rangle=A_{\frak{q}}$ is flat.

It is enough to check flatness at the maximal ideals of $A$, so we have finished.
\end{proof}

As we saw at the start of the appendix, there is a natural isomorphism \[\frak{C}(A/B)\isoto\Hom_B(A,B),\quad x\mapsto \Tr_{F/M}(x\,\cdot).\] For any maximal ideal $\frak{m}\lhd B$, the localisation $A_{\frak{m}}$ is a flat (by the previous lemma), hence free, $B_{\frak{m}}$-module of rank $n=|F:M|$; the importance of this is that it implies, using the above isomorphism, that $\frak{C}(A/B)_{\frak{m}}$ is a free $B_{\frak{m}}$-module of rank $n$.

Let us now temporarily suppose that $A$, $B$ are both {\em local} rings. Then we may write \[A=B[T_1,\dots,T_m]_{\frak{n}}/\langle f_1,\dots,f_m\rangle,\] where $\frak{n}$ is a maximal ideal of $B[T_1,\dots,T_m]$ which sits over $\frak{m}_B$, and where $f_1,\dots,f_m\in \frak{n}$ is a regular sequence. In this case, we have:

\begin{lemma}
Suppose that $A$, $B$ are local rings. Let $J\in A$ be the determinant of the Jacobian matrix $\left(\dd{f_i}{T_j}\right)_{i,j}$. Then \[\frak{C}(A/B)=J^{-1}A.\]
\end{lemma}
\begin{proof}
First note that since $F/M$ is separable, the extension $A/B$ is generically smooth and therefore $J\neq 0$. For any $y\lhd\!^1B$, it is clear that $\frak{C}(A_y/B_y)=\frak{C}(A/B)A_y$ where $A_y=(B\setminus y)^{-1}A$, which is a Dedekind domain. A standard formula for finite extensions of Dedekind domains \cite{Neukirch1999} states \[\frak{C}(A_y/B_y)=\prod_{0\neq Y\lhd A_y}Y^{-d_{Y/y}},\] where $d_{Y/y}=\nu_Y(\frak{D}(\hat{A}_Y/\hat{B}_y))$ (here $\nu_Y$ denotes the discrete valuation on $F_Y$). But by theorem \ref{theorem_different_and_jacobian_for_cdvr}, $d_{Y/y}=\nu_Y(J)$. Therefore \[\frak{C}(A_y/B_y)=J^{-1}A_y.\] 

According to the remarks above, both $\frak{C}(A/B)$ and $J^{-1}A$ are free $B$-modules of rank $n=|F:M|$, so if we pick a basis for $F$ as a vector space over $M$ and interpret both these free modules as submodules of $M^n$, then there exists a $\tau\in GL_n(M)$ such that $\tau\frak{C}(A/B)=J^{-1}A$. However, for any height one-prime $y\lhd B$, we have just seen that $\frak{C}(A/B)A_y=J^{-1}A_y$, implying that $\tau\in GL_n(B_y)$. Since $B$ was assumed to be normal, $B=\bigcap_{y\lhd\!^1B}B_y$ and so $\tau\in GL_n(B)$; therefore $\tau\frak{C}(A/B)=\frak{C}(A/B)$, which completes the proof.
\end{proof}

Let us now continue to suppose that $B$, but not necessarily $A$, is local. Being a finite extension of the local ring $B$, $A$ is at least a semi-local ring (i.e. has only finitely many maximal ideals), and $A\otimes_B\hat{B}=\bigoplus_{\frak{m}}\hat{A_{\frak{m}}}$, where $\frak{m}$ varies over the maximal ideals of $A$. Each completion $\hat{A_{\frak{m}}}$ is a finite extension of $\hat{B}$.

Since $\hat{B}$ is flat over $B$, we have
\begin{align*}
\Hom_B(A,B)\otimes_B\hat{B}
	&=\Hom_{\hat{B}}(A\otimes_B\hat{B},\hat{B})\\
	&=\Hom_{\hat{B}}(\bigoplus_{\frak{m}}\hat{A_{\frak{m}}},\hat{B})\\
	&=\bigoplus_{\frak{m}}\Hom_{\hat{B}}(\hat{A_{\frak{m}}},\hat{B}).
\end{align*}
Replacing each module of homomorphisms by the relevant codifferent gives us a natural isomorphism of $\hat{A}=A\otimes_B\hat{B}$-modules: \begin{equation}\frak{C}(A/B)\otimes_B\hat{B}\cong\bigoplus_{\frak{m}} \frak{C}(\hat{A_{\frak{m}}}/\hat{B}).\label{9}\end{equation} Now fix one such maximal ideal $\frak{m'}$ of $A$. For any other $\frak{m}\neq\frak{m}'$, we have $\frak{C}(\hat{A_{\frak{m}}}/\hat{B})\otimes_{\hat{A}}\hat{A_{\frak{m'}}}=0$. Indeed, there exists $s\in\frak{m'}\setminus\frak{m}$; then multiplication by $s$ is an automorphism of the finitely generated $\hat{A}$-module $\frak{C}(\hat{A_{\frak{m}}}/\hat{B})$, so Nakayama's lemma implies $\frak{C}(\hat{A_{\frak{m}}}/\hat{B})\otimes_{\hat{A}}\hat{A_{\frak{m'}}}=0$. Tensoring (\ref{9}) by $-\otimes_A A_{\frak{m'}}$ now reveals that $\frak{C}(A/B)\otimes_A\hat{A_{\frak{m'}}}=\frak{C}(\hat{A_{\frak{m'}}}/\hat{B})$.

We are now sufficiently prepared to prove our main `different$=$Jacobian' result; we return to the setting of having a finite extension $A/B$ of excellent, normal rings which is a local complete intersection:

\begin{proof}[Proof of theorem \ref{theorem_different_and_jacobian_in_general}]
We use the notation from the statement of the theorem. Let $\frak{m}$ be a fixed maximal ideal of $A$; we are free to localise $B$ at $B\cap \frak{m}$, so henceforth we will assume $B$ is local. Then \[\hat{A_{\frak{m}}}=\hat{B}[T_1,\dots,T_m]_{\frak{n}}/\langle f_1,\dots,f_m\rangle\] represents $\hat{A_{\frak{m}}}$ as a finite complete intersection over $\hat{B}$ and so we may apply the previous lemma to deduce $\frak{C}(\hat{A_{\frak{m}}}/B)=J_{\frak{m}}^{-1}\hat{A_{\frak{m}}}$.

By the discussion above ($\frak{m}'$ is now $\frak{m}$), $\frak{C}(\hat{A_{\frak{m}}}/B)=\frak{C}(A/B)\otimes_A\hat{A_{\frak{m}}}$, and so from the faithful flatness of $\hat{A_{\frak{m}}}$ over $A_\frak{m}$, we deduce $\frak{C}(A/B)_{\frak{m}}=J_{\frak{m}}^{-1}A_\frak{m}$, as required.
\end{proof}
\end{appendix}

\end{document}